\documentclass[11pt]{amsart}

\usepackage{amsmath,amsthm,amsfonts,amssymb,mathrsfs}
\usepackage{inputenc,mathrsfs}

\numberwithin{equation}{section}
\usepackage[dvips]{graphicx}
\usepackage[colorlinks]{hyperref}

\newtheorem{definition}{Definition}[section]
\newtheorem{theorem}{Theorem}[section]

\newtheorem{lemma}{Lemma}[section]
\newtheorem{corollary}{Corollary}[section]

\theoremstyle{remark}
\newtheorem{remark}{Remark}[section]

\DeclareMathOperator{\diam}{\mathrm{diam}}
\DeclareMathOperator{\riem}{\mathrm{Rm}}
\DeclareMathOperator{\ric}{\mathrm{Ric}}
\DeclareMathOperator{\hess}{\mathrm{Hess}}

\DeclareMathOperator{\vol}{\mathrm{Vol}}
\DeclareMathOperator{\tr}{\mathrm{tr}}
\newcommand{\de}{\,\mathrm{d}}

\newcommand{\R}{\mathbb{R}}
\newcommand{\dist}{\mathrm{dist}}

\def\uclhome{@ucl.ac.uk}

\author{Panagiotis Gianniotis}
\thanks{}
\address{Panagiotis Gianniotis: Department of Mathematics, University of Toronto, 40 St George Street, Toronto, ON M5S 2E4, Canada}
\curraddr{}
\email{p.gianniotis@utoronto.ca}

\author{Felix Schulze}
\thanks{}
\address{Felix Schulze: 
  Department of Mathematics, University College London, 25 Gordon St,
  London WC1E 6BT, UK}
\curraddr{}
\email{f.schulze\uclhome}

\subjclass[2000]{}

\dedicatory{}


\begin{document}
\title[Ricci flow from spaces with isolated conical singularities]{Ricci flow from spaces with isolated conical singularities}
\maketitle
\begin{abstract} Let $(M,g_0)$ be a compact $n$-dimensional Riemannian manifold with a finite number of singular points, where the metric is asymptotic to a non-negatively curved cone over $(\mathbb{S}^{n-1},g)$. We show that there exists a smooth Ricci flow starting from such a metric with curvature decaying like C/t. The initial metric is attained in Gromov-Hausdorff distance and smoothly away from the singular points. In the case that the initial manifold has isolated singularities asymptotic to a non-negatively curved cone over $(\mathbb{S}^{n-1}/\Gamma,g)$, where $\Gamma$ acts freely and properly discontinuously, we extend the above result by showing that starting from such an initial condition there exists a smooth Ricci flow with isolated orbifold singularities. 
\end{abstract}
\section{Introduction}
Consider a smooth solution $(M,g(t))_{t\in [0,T)}$ to the Ricci flow
$$\frac{\partial}{\partial t} g = -2 \ric(g)\, ,$$
starting from a closed Riemannian manifold $(M,g(0))$. Hamilton has shown in \cite{Hamilton82} that the existence time $T$ of the unique maximal solution is bounded from below by $C/K$, where $C=C(n)>0$ and $K=\sup_{M}|\riem(g(0))|$.  

It is a natural question to ask which non-smooth spaces can arise as initial data for smooth solutions to the Ricci flow. In \cite{Simon09, Simon12, Simon14}, Simon shows that one can construct a smooth Ricci flow starting from a space that can be approximated by a sequence of smooth 3-dimensional manifolds that  is  locally uniformly non-collapsed and has curvature operator locally uniformly  bounded from below. This result has been applied by Lebedeva--Matveev--Petrunin--Shevchishin \cite{LMPS15} to show that 3-dimensional polyhedral manifolds with nonnegative curvature in the sense of Alexandrov can be approximated by non-negatively curved 3-dimensional Riemannian manifolds. Koch--Lamm \cite{KochLamm12} show that from any initial metric, which is a small $L^\infty$-perturbation of the standard Euclidean metric, there exists a smooth solution to Ricci--DeTurck flow. This was extended in \cite{KochLamm15} to small $L^\infty$-perturbations of a $C^2$ background metric on a uniform $C^3$ manifold. We note that small $L^\infty$-perturbations allow for conical singularities where the cones are sufficiently close to Euclidean space. 

Much more is known in dimension two. The results of Simon are still valid, and the work of Giesen--Topping and Topping \cite{GiesenTopping11,Topping15} implies that given any initial data, even incomplete with unbounded curvature, there exists a smooth Ricci flow that becomes complete for $t>0$,  which is unique in an appropriate class. Moreover,  Yin \cite{Yin10,Yin13} and Mazzeo--Rubinstein--Sesum \cite{MazzeoRubinsteinSesum15} consider two dimensional Ricci flows that preserve the conical singularity. For a generalisation to higher dimensions of Ricci flows that preserve a certain class of singularities see the work of Vertman \cite{Vertman16}. In the case of K\"ahler-Ricci flow also more is known. Short-time existence from non-smooth initial data was studied by Guedj--Zeriahi \cite{GuedjZeriahi13}, Di Nezza--Lu \cite{DiNezzaLu17} and Song--Tian \cite{SongTian09}, where the last article also treats the evolution through singularities. Preserving conical singularities in the K\"ahler case was considered by Chen--Wang \cite{ChenWang15}.

In this paper we consider smooth Ricci flows that start from compact smooth initial spaces $(Z,g_Z)$ with isolated conical singularities. Such spaces can be expected to arise as the limiting space of a smooth Ricci flow $(N,h(t))_{t \in [0,T)}$ as $t \rightarrow T$, as the following heuristic argument describes. Assume that at $(p,T)$ the flow has a type I singularity. By work of Naber \cite{Naber10}, Enders--M\"uller--Topping \cite{EndersMuellerTopping11} and Mantegazza--M\"uller \cite{MantegazzaMueller15} it is known that any parabolic blow-up of the flow around $(p,T)$ converges to a smooth, shrinking, non-trivial, gradient soliton solution. Furthermore, if one assumes that this soliton is non-compact and the Ricci curvature goes to zero at infinity, then it is known by work of Munteanu--Wang \cite{MunteanuWang14} that the gradient shrinking soliton is smoothly asymptotic to a cone over a compact Riemannian manifold. Assuming further that such a tangent flow is unique, i.e.~does not depend on the sequence of rescalings chosen, it should be possible to show that $(N,h(t))$ converges to a smooth space $(Z,g_Z)$ with an isolated conical singularity. We would then like to continue the flow so that it immediately becomes smooth after time $T$. For an example of such a behaviour on the level of soliton solutions, see the work of Feldman--Ilmanen--Knopf \cite{FeldmanIlmanenKnopf03}. We note furthermore that such a picture of a smooth limiting space with isolated conical singularities can be made precise for mean curvature flow.

We define a compact Riemannian manifold with isolated conical singularities as follows.

\begin{definition}\label{conical_sing}
We say that $(Z,g_Z)$ is a compact space with isolated conical singularities at $\{z_i \}_{i=1}^Q \subset Z$  modelled on the cones 
$$(C(X_i), g_{c,i}=dr^2+r^2 g_{X_i}),$$
 where $(X_i,g_{X_i})$ are smooth compact Riemannian manifolds, if:  
\begin{enumerate}
\item $(Z\setminus\{z_1,\ldots,z_Q\}, g_Z)$ is a smooth Riemannian manifold.
\item The metric completion of  $(Z\setminus\{z_1,\ldots,z_Q\}, g_Z)$ is a compact metric space $(Z,d_Z)$.
\item There exist maps $\phi_i : (0,r_0] \times X_i \rightarrow Z\setminus \{z_1,\ldots,z_Q\}$, $i=1,\ldots,Q$, diffeomorphisms onto their image, such that $\lim_{r \rightarrow 0} \phi_i(r,p) = z_i$ for any $p \in X_i$ and
\begin{align}
\sum_{j=0}^4  r^j |(\nabla^{g_{c,i}})^j (\phi_i^* g_Z - g_{c,i})|_{g_{c,i}} &<k_Z(r), \label{5.conical}
\end{align}
for some function $k_Z:(0,r_0]\rightarrow \mathbb R^+$ with $\lim_{r\rightarrow 0} k_Z(r) = 0$.
\end{enumerate}
\end{definition}

We prove the following short-time existence result.

\begin{theorem}\label{stexist}
Let $(Z,g_Z)$ be a compact Riemannian manifold with isolated conical singularities at $\{z_i\}_{i=1}^Q\subset Z$, each modelled on a cone
$$(C(\mathbb{S}^{n-1}), g_{c,i}= dr^2+r^2 g_i)$$  with $\riem(g_i)\geq1$, but $\riem(g_i)\not\equiv 1$.

Then, there exists a smooth manifold $M$, a smooth Ricci flow $(g(t))_{t\in (0,T]}$ on $M$  and a constant $C_{\riem}$ with the following properties. 
\begin{enumerate}
\item  $(M,d_{g(t)}) \rightarrow (Z,d_Z)$ as $t\rightarrow 0$, in the Gromov--Hausdorff topology.
\item There exists a map $\Psi: Z\setminus \{z_1,\ldots,z_Q\} \rightarrow M$, diffeomorphism onto its image, such that
$\Psi^* g(t)$ converges to $g_Z$, smoothly uniformly away from $z_i$, as $t\rightarrow 0$.
\item $\max_M |\riem(g(t))|_{g(t)}\leq C_{\riem}/t$  for $t\in (0,T]$.
\item Let $t_k\searrow 0$ and $p_k \in (\mathrm{Im}\,\Psi )^c\subset (M,d_{g(t_k)})$. Suppose that $p_k\rightarrow z_i$ under the Gromov--Hausdorff convergence, as $k \rightarrow \infty$. Then 
\begin{equation*}
(M,t_k^{-1}g(t_k t),p_k)_{t\in (0,t_k^{-1}T]}\rightarrow (N_i,g_{e,i}(t),q)_{t\in (0,+\infty)},
\end{equation*}
where $(N_i,g_{e,i}(t))_{t\in (0,+\infty)}$ is the Ricci flow induced by the unique expander $(N_i,g_{N_i},f_i)$ with positive curvature operator that is asymptotic to the cone $(C(\mathbb S^{n-1}),g_{c,i})$.
\end{enumerate}
\end{theorem}

To construct the solution, we desingularise the initial metric by glueing in expanding gradient solitons with positive curvature operator, each asymptotic to the cone at the singular point, at a small scale $s$. These expanding solitons exist due to a recent result of Deruelle \cite{Deruelle16}.  Localising a recent stability result of Deruelle--Lamm \cite{DeruelleLamm16} for such expanding solutions, we show that there exists a solution from the desingularised initial metric for a uniform time $T>0$, with corresponding estimates, independent of the glueing scale $s$. The solution is then obtained by letting $s \rightarrow 0$. 

The last point in the statement of the above theorem says that the limiting solution has the corresponding expanding gradient soliton as a forward tangent flow at each initial singular point. We further note that our construction doesn't require that the initial data or the constructed approximating sequence satisfy any lower bound on the curvature.  Moreover, aside from the existence of the expanding gradient solitons and the stability result of Deruelle--Lamm, the construction does not depend in any way on the non-negativity assumption on the curvature of the conical models.

In case that the isolated singularities are modelled on cones over a quotient of $(\mathbb{S}^{n-1},\bar{g})$ with $\riem(\bar{g})\geq 1$ we can show that there exists a smooth solution to the orbifold Ricci flow starting from such a space, with isolated orbifold points.  Each initial cone $(C(\mathbb{S}^{n-1}/\Gamma_i), dr^2+r^2 g_i)$, with $\Gamma_i$ non-trivial, corresponds to an isolated orbifold point in the flow.

\begin{theorem}\label{orbifold_thm}
Let $(Z,g_Z)$ be as in Theorem \ref{stexist}, with singularities at $\{z_i\}_{i=1}^Q$ modelled on cones $(C(\mathbb{S}^{n-1}/\Gamma_i), g_{c,i}:=dr^2+r^2 g_i)$ with $\riem(g_i)\geq1$,  $\riem(g_i)\not\equiv 1$, and $\Gamma_i$  acting freely and properly discontinuously.

Then there exists a smooth \emph{orbifold} Ricci flow $(M,g(t))_{t\in (0,T]}$ with isolated orbifold singularities, each modelled on $\mathbb R^n / \Gamma_i$, and a constant $C_{\riem}$ for which (1)-(3) of Theorem \ref{stexist}  hold. Moreover
\begin{itemize}
\item[($4'$)]  Let $t_k\searrow 0$ and $p_k \in (\mathrm{Im}\,\Psi )^c\subset (M,d_{g(t_k)})$. Suppose that $p_k\rightarrow z_i$ under the Gromov--Hausdorff convergence, as $k \rightarrow \infty$. Then 
\begin{equation*}
(M,t_k^{-1}g(t_k t),p_k)_{t\in (0,t_k^{-1}T]}\rightarrow (\mathcal O_i,g_{e,i}(t))_{t\in (0,+\infty)},
\end{equation*}
where $(\mathcal O_i,g_{e,i}(t))_{t\in (0,+\infty)}$ is the \emph{orbifold} Ricci flow induced by the unique \emph{orbifold quotient} expander $(\mathcal O_i,g_{\mathcal O_i},f_i)$ with positive curvature operator that is asymptotic to the cone $(C(X_i),g_{c,i})$.
\end{itemize}
\end{theorem}

The proof of Theorem \ref{orbifold_thm} is a direct modification of the proof of Theorem \ref{stexist}. We do this by showing that there exists a unique orbifold quotient expander $(\mathcal O_i,g_{\mathcal O_i},f_i)$ with positive curvature operator and one isolated orbifold point that is asymptotic to the cone $(C(X_i),g_{c,i})$: see Theorem \ref{thm:quotient_expander}. 

We can also allow for cones as models for the singularities which are not non-negatively curved, provided they are small perturbations of non-negatively curved cones considered in Theorem \ref{stexist}.

\begin{theorem}\label{perturbation_thm}
Let $(Z,g_Z)$ be as in Theorem \ref{stexist}, with singularities at $\{z_i\}_{i=1}^Q$ modelled on cones $(C(\mathbb S^{n-1}), g_{c,i}:=dr^2+r^2 g_i)$. Let $(N,g_{N_i},f_i)$ be expanders with positive curvature operator asymptotic to $(C(\mathbb S^{n-1}),g_{c,i}'=dr^2+r^2 g_i ')$ with $\riem(g_i ')\geq 1$, $ \riem(g_i')\not\equiv 1$. Then there exist $\varepsilon_i>0$ depending on $g_{N_i}$ such that if 
\begin{equation*}
|(\nabla^{g_i})^j (g_i '-g_i)|_{g_i} <\varepsilon_i\, ,
\end{equation*} 
where $0\leq j \leq 4$, then there exists a smooth Ricci flow $(M,g(t))_{t\in (0,T]}$ and $C_{\riem}$ for which (1)-(3) of Theorem \ref{stexist} hold.
\end{theorem}

Of course the analogous statement is also true for the orbifold case of Theorem \ref{orbifold_thm}.  We would like to point out that the condition that the curvature operator of the cones $(C(\mathbb S^{n-1}),g_{c,i}'=dr^2+r^2 g_i ')$ is non-negative is not preserved under small perturbations. This implies that the curvature operator of $(Z,g_Z)$ might be unbounded from below in a neighborhood of the singular points. In this case, the constructed flow $(M,g(t))_{t\in (0,T]}$ will have curvature operator unbounded from below as $t\searrow 0$. 

Observe also that the case $\riem(g_i)\equiv 1$ in Theorems \ref{stexist} and \ref{orbifold_thm} corresponds to a smooth Riemannian manifold or orbifold respectively and there is nothing to prove. Similarly, the case $\riem(g_i')\equiv 1$ in Theorem \ref{perturbation_thm} corresponds to initial data which are perturbations of a smooth Riemannian metric, which is dealt with by Koch--Lamm in \cite{KochLamm15}.

{\bf Outline.} In Section \ref{prelim} we recall some facts about gradient Ricci expanders asymptotic to cones and introduce notation. 

In Section \ref{conical metrics} we define the class of Riemannian manifolds $\mathcal{M}(\eta,\Lambda,s)$, which can be understood as a local smoothing of an isolated conical singularity with an expander at scale $s$. In Theorem \ref{main_thm} we state local a priori curvature estimates for Ricci flows with initial data in $\mathcal{M}(\eta,\Lambda,s)$, which are uniform in $s$. To prove these estimates we separate the initial manifold in the \emph{conical} and \emph{expanding region}. The idea is then to use Perelman's pseudocality theorem  to control the flow for a short time in the conical region, showing that it remains conical, and use a localised version of the stability result of Deruelle--Lamm \cite{DeruelleLamm16} to control the flow in the expanding region. However, to exploit the latter we need to work with the Ricci--DeTurck flow for a suitably chosen background metric, which is an interpolation of the initial metric and the expanding metric at scale $s+t$. To pass from a solution of Ricci flow to the corresponding solution to Ricci-DeTurck flow one needs to pull back by the inverse of a solution to harmonic map heat flow $\psi$ with the background metric as a target. Assuming an a priori bound on $|\nabla \psi|$, we use Perelman's pseudolocality theorem to control the solution to Ricci--Deturck flow in the conical region: Lemma \ref{conical_region}. Then, localising the stability result of Deruelle--Lamm we control the Ricci--DeTurck flow in the expanding region: Lemma \ref{expanding_region}. We finally show how these results can be combined to prove Theorem \ref{main_thm}. A central point  is that a posteriori the assumed threshold for $|\nabla \psi|$ is never achieved, and thus the argument closes. 

In Section \ref{main_estimates} we give the proofs of Lemma \ref{conical_region} and Lemma \ref{expanding_region}. This includes a `pseudolocality' theorem for the harmonic map heat flow: Lemma \ref{hmf_pseudolocality}; and the localisation of the stability result of Deruelle--Lamm: Lemma \ref{long_time_estimate}.

In Section \ref{flowing_sing} we give the proof of Theorem \ref{stexist}, as well as that of Theorem \ref{perturbation_thm}. In Section \ref{approx} we construct the approximation sequence, by glueing in the expander metric at scale $s$ into $g_Z$ around the singular point, and showing that this metric is in the class  $\mathcal{M}(\eta,\Lambda,s)$. The proof of the statements of Theorem \ref{stexist} then follows in Sections \ref{almost_cone}-\ref{tangent flow}. In Section \ref{final proofs} we show how the proof of Theorem \ref{stexist} can be modified to prove  Theorem \ref{perturbation_thm}. Finally,  in Section \ref{orbifold expanders} we show the existence of orbifold quotient expanders and prove Theorem \ref{orbifold_thm}.

{\bf Acknowledgments:} The authors wish to thank Alix Deruelle for many interesting discussions on expanding Ricci solitons.\\

\section{Preliminaries}\label{prelim}
\subsection{Expanders asymptotic to cones.}  \label{asymp_conical_exp} A triple
$(N,g_N,f)$, where $(N,g_N)$ is a Riemannian manifold and $f$ a smooth function on $M$, is said to be a gradient Ricci expander if it satisfies the equation
\begin{eqnarray}\label{soliton_eqn}
\hess_{g_N} f=\frac{1}{2}\mathcal L_{\nabla f} g_N=\ric(g_N)+\frac{g_N}{2}.
\end{eqnarray}
As a consequence, the well known formula 
\begin{equation*}
|\nabla f |^2=f+c-R,
\end{equation*}
holds, for an appropriate constant $c$. 

Note that $f$ is well defined up to a constant and linear function. Hence, provided $(N,g_N)$ has bounded curvature, we will assume w.l.o.g. that $c=\inf_M R:=R_{inf}$, where $R$ denotes the scalar curvature. Such a normalisation always ensures that $f\geq 0$.

A gradient Ricci expander generates a solution to Ricci flow, which moves only by diffeomorphisms and scaling:
Let $\varphi_t$, $t>0$,  be the diffeomorphisms satisfying the ODE 
\begin{eqnarray}\label{diffeo_evol}
\frac{\de}{\de t}\varphi_t&=&-\frac{1}{t} \nabla f\circ \varphi_t, \\
\varphi_1 &=& id_N.
\end{eqnarray}
Then the family $g_e(t)=t \varphi_t^* g_N$ solves Ricci flow for $t>0$. Define $f_s=f\circ \varphi_s$, for any $s>0$.

We note for later reference that the ODE implies that
\begin{equation}
\varphi_{s}\circ \varphi_{t} = \varphi_{st}\, .
\end{equation}

Let $(X,g_X)$ be a smooth Riemannian manifold and $$(C(X), g_c=dr^2+r^2 g_X,o)$$ be the associated cone with vertex $o$. We will say that the expander $(N,g_N,f)$ is asymptotic to the cone $C(X)$ if
\begin{enumerate}
\item There is a diffeomorphism onto its image $F: [\Lambda_0,\infty)\times X \rightarrow N$ such that $N\setminus \mathrm{Im}\,(F)$ is compact and
\begin{eqnarray}
f(F(r,q))=\frac{r^2}{4}, \nonumber
\end{eqnarray}
for every $(x,q)\in [\Lambda_0,\infty)\times X$.
\item  $\displaystyle{\sum_{j=0}^4} \sup_{\partial B_{g_c}(o,r)} r^j|( \nabla^{g_c})^j (F^* g_N - g_c  ) |_{g_c}=k_{exp}(r)$, where $\displaystyle{\lim_{r\rightarrow\infty}}k_{exp}(r)=0$.
\end{enumerate}
From \cite[Theorem 3.2]{Deruelle14} we may assume w.l.o.g. that 
\begin{equation}
F(r,q)=J_{\frac{r^2}{4}-\frac{\Lambda_0^2}{4}}(F(\Lambda_0,q)), \label{flow}
\end{equation}
where $J_t:N\rightarrow N$ is the flow of the vector field $\nabla f /|\nabla f|^2$ with $J_0=id_N$.

A natural radial coordinate at infinity on the expander is given by $$\mathbf{r}:=2\sqrt{f}= (F^{-1})^* r\, .$$ 
 Similarly, for the expander at scale $s$, it will be convenient to consider the  radial coordinate at infinity defined as $\mathbf r_s=2 \sqrt{sf_s}$.
 
 In fact, if we define $F_s: [\Lambda_0\sqrt s,\infty) \times X \rightarrow N$ by $F_s=\varphi_s^{-1} \circ F \circ a_s$, where $a_s(r,q)=(\frac{r}{\sqrt s},q)$ for $(r,q)\in [0,\infty)\times X$, it follows that  $\mathbf r_s(F_s(r,q))=r$.
 
Moreover, since
\begin{equation}\label{conv_to_cone}
\begin{aligned}
 r^j|(\nabla^{g_c})&^j(F_s^* g_e(s) - g_c) |_{g_c} (r,q)\\
=\ &r^j | (\nabla^{g_c})^j (a_s^*\circ F^* \circ (\varphi_s^{-1})^* g_e(s) -g_c ) |_{g_c} (r,q),\\
=\ &r^j a_s^* ( | (\nabla^{(a_s^{-1})^* g_c})^j (F^* (s g_N) - (a_s^{-1})^* g_c)|_{(a_s^{-1})^* g_c}(r,q),\\
=\ &r^j   | (\nabla^{s g_c})^j (F^* (s g_N) - s g_c)|_{s g_c}   (r s^{-1/2},q),\\
=\  &\big(rs^{-1/2}\big)^j|(\nabla^{g_c})^j(F^* g_N -g_c) |_{g_c} (r s^{-1/2},q) = k_{exp}(r s^{-1/2}),
\end{aligned}
\end{equation}
$F_s^* g_e(s)$ converges to $g_c$ as $s\rightarrow 0$, uniformly away from $o$ in $C^4_{loc}$. Moreover $|\nabla^{g_e(s)} \mathbf r_s |_{g_e(s)}\rightarrow 1$, uniformly away from $o$.

We will also need the following lemma, whose proof we postpone until Section \ref{main_estimates}.

\begin{lemma}\label{exp_con_est}
Let $(N, g_N, f)$  be an asymptotically conical gradient Ricci expander and let $(g_0(t))_{t\geq 0}$ be the induced Ricci flow with $g_0(0)=g_N$. There exists $\gamma_0\geq 1$ and $C,\Lambda_0>0$ such that
\begin{gather*}
 | F^* g_0  - g_c|_{g_c} + \mathbf r| \nabla^{g_c} F^* g_0|_{g_c}<\frac{1}{100},\qquad
 \frac{1}{2}\leq|\nabla^{g_0} \mathbf r |_{g_0}\leq 2,\\
|\mathbf r \Delta_{g_0} \mathbf r |\leq 4(n-1), \qquad \mathbf r^2 |\riem(g_0)|_{g_0} \leq C(g_c)
\end{gather*}
in $\big\{ (x,t)\in N\times [0,+\infty),\; \mathbf r(x) \geq \sqrt{\gamma_0 t + \Lambda_0^2} \big\}$.
\end{lemma}

From now on, given any expander asymptotic to a cone, $\gamma_0$ and $\Lambda_0$ will refer to the constants given by Lemma \ref{exp_con_est}.
\subsection{Expanders asymptotic to cones with positive curvature operator.}$\ $\\[1ex]
It is known by the recent work of Deruelle \cite{Deruelle16} that, given $(\mathbb{S}^{n-1},g)$ with $\text{Rm}(g)\geq 1$, there exists a unique expanding gradient soliton $(N,g_N,f)$ with non-negative curvature operator, which is asymptotic to the cone $$(C(\mathbb{S}^{n-1}), dr^2+r^2g, o).$$ We note the following consequence.  The proof has similarities to the argument of Perelman in the proof of Claim 2 in \cite[\S 12]{Perelman02}.

\begin{lemma}\label{lem:expander_pos}
Assume that $(\mathbb{S}^{n-1},g)$ satisfies $\riem(g)\geq 1$ but $\riem(g)\not\equiv 1$. Then the expander $(N,g_N,f)$ that is asymptotic to $(C(\mathbb{S}^{n-1}), dr^2+r^2g, o)$, given by \cite{Deruelle16}, has positive curvature operator.  Moreover, if $f$ is normalised so that  $|\nabla f|^2=f+R_{inf}-R$, then it is unique.
 \end{lemma}

\begin{proof}
 Assume that there exists a point $p \in N$ such that $\text{Rm}(g_N)(p)$ has a zero eigenvalue. By Hamilton's strong maximum principle there exists $\delta>0$ such that, for every $t\in (0,\delta]$, $\mathrm{Ker}(\riem(g_e(t)))$ is a positive rank subbundle of $\Lambda^2 T^* N$, invariant under parallel translation.
 
Consider $(1,q)\in (0,+\infty)\times \mathbb S^{n-1}$ such that $\riem(g)(q)>1$ and let $g_c=dr^2 + r^2 g$. Then $\mathrm{Ker}(\riem(g_c))$ in a neighbourhood of $(1,q)$ consists solely of elements of the form $\partial_r \wedge V$, for $V\in T\mathbb S^{n-1}$. Moreover, recall that for $W\in  T_q\mathbb S^{n-1}$ 
\begin{equation}
\nabla_W (\partial_r \wedge V)|_{(1,q)}= W\wedge V +\partial_r \wedge \nabla_W V,\label{cov_deriv}
\end{equation}
since $\nabla_W \partial_r=\frac{1}{r} W$ on the cone.

Now,  since $F_t^* g_e(t)$ converges to $g_c$  as $t\rightarrow 0$, we conclude that around $(1,q)$ there is a section $\partial_r\wedge V$ of $\mathrm{Ker}(\riem(g_c))$ satisfying $\nabla (\partial_r\wedge V)|_{(1,q)}=0$. This contradicts \eqref{cov_deriv}.
 
 To prove uniqueness of $f$, note that any other potential function $\tilde f$ will satisfy $\hess_{g_N}(f-\tilde f)=0$. This implies that either $\tilde f=f+ c$, for some constant $c$, or the expander splits a line by DeRham's theorem (recall that $N$ is simply connected by construction). However, the latter is not possible, since the unique tangent cone at infinity would split a line, contradicting that it has an isolated singularity. But then, if $\tilde f$ satisfies  $|\nabla \tilde f|^2=\tilde f+R_{inf}-R$, we see that $c=0$.
\end{proof}

\subsection{Distance distortion estimate.}

Let $g_0(t)=g_e(t+1)$ denote the associated Ricci flow with $g_0(0)=g_N$. Since $(g_0(t))_{t\geq 0}$ is a Type III solution for the Ricci flow, namely
\begin{eqnarray*}
\max_M |\riem(g_0(t))|_{g_0(t)}\leq \frac{C}{t+1},
\end{eqnarray*}
the following distance distortion estimate holds. There exists $C(g_N)>0$ such that for every $x,y\in N$, $t\geq0$
\begin{equation}
d_{g_0(0)}(x,y) -C(g_N) \sqrt{t} \leq d_{g_0(t)}(x,y). \label{dist_dist}
\end{equation} 
This estimate is due to Hamilton, for a proof see for example \cite[Lemma 8.33]{ChowLuNi06}.

\section{Flowing almost conical metrics}\label{conical metrics}
In this section we fix an asymptotically conical gradient Ricci expander $(N,g_N,f)$ with positive curvature operator and let $\gamma_0$, $\Lambda_0$ be as in Lemma \ref{exp_con_est}. We will consider the following class of Riemannian manifolds as initial data for the Ricci flow.  Recall that a natural coordinate at infinity for an expander at scale $s$ is given by $\mathbf r_s=2 \sqrt{sf_s}$.
\begin{definition} \label{almost_conical}
Given $\eta, s>0$, $\Lambda \geq \Lambda_0$  define the class $\mathcal M(\eta,\Lambda,s)$ of complete Riemannian manifolds $(M,g)$ with bounded curvature satisfying the following: there exist $\Phi_s:\left\{\mathbf r_s \leq 1\right\} \rightarrow M$, a function $r_s: \mathrm{Im}\, \Phi_s\rightarrow [\Lambda\sqrt{s},1]$ defined by
\begin{equation*}
 r_s=\max\left\{ (\Phi_s^{-1})^* \mathbf r_s, \Lambda \sqrt s \right\}
 \end{equation*}
and 
\begin{itemize}
\item $\sum_{j=0}^4 r^j|(\nabla^{g_c})^j((\Phi_s\circ F_s)^* g -g_c)|_{g_c} + r^j|(\nabla^{g_c})^j( F_s^*g_e(s) -g_c)|_{g_c} <\eta $ \\ in $[\Lambda\sqrt s,1]\times X$.
\item $|\Phi_s^* g - g_e(s)|_{g_e(s)}<\eta,\text{ in }\left\{ \mathbf r_s \leq 2(\Lambda+1)\sqrt s\right\}.$
\end{itemize}
Note that $[(\Phi_s\circ F_s)^* r_s ](r,q)=r$ in $[\Lambda \sqrt s,1]\times X$.
\end{definition}

A metric in $\mathcal M(\eta,\Lambda,s)$ can be viewed as the smoothening of an isolated conical singularity with an expander at scale $s$. The function $r_s$ behaves like the distance from the origin of the cone $C(X)$ when $\eta$ is small. The parameter $\Lambda$ separates the manifold  into two regions, the \textit{conical region} where it is $\eta$-close to the cone, and the \textit{expanding region} where it is $\eta$-close to the expander at scale $s$.  

The aim is to prove a priori curvature estimates for Ricci flows with initial data in $\mathcal M(\eta,\Lambda,s)$, which are uniform in $s$.  

\begin{theorem}\label{main_thm}
\mbox{}
\begin{enumerate}
\item Given $\Lambda>0$, there exist $\eta_0(g_N), s_0(\Lambda),  \tau_0(g_N), C(g_N)$ such that for every $s\in (0,s_0]$ the following holds: \\
If $(M,g(t))_{t\in [0,T]}$ is a complete Ricci flow with bounded curvature, and $(M,g(0))\in\mathcal M(\eta_0, \Lambda, s)$ then 
\begin{equation*}
\begin{aligned}
\max_{\{r_s \leq \frac{3}{4}\}} |\riem(g(t))|_{g(t)}&\leq \frac{C}{t}, \textrm{ for }t\in (0,\min\{\tau_0,T\}],\\ 
\max_{\{r_s \leq  \frac{3}{4}\}} \sum_{j=0}^2 r_s^{2+j} |(\nabla^{g(t)})^j\riem(g(t))|_{g(t)} &\leq C, \textrm{ for } t\in [0,\min\{\tau_0,T\}].
\end{aligned} 
\end{equation*}

\item  For every $\varepsilon>0$ and  integer $k\geq 0$ there exist $\eta_1=\eta_1(g_N,\varepsilon,k) $, $\gamma_1=\gamma_1(g_N,\varepsilon,k)$ such that if $s\in (0,s_0]$ and $\gamma\geq\gamma_1$ then the following holds: 

  If $(M,g(t))_{t\in [0,T]}$ is a complete Ricci flow with bounded curvature and $(M,g(0))\in \mathcal M(\eta_1,\Lambda,s)$, then for every $t\in (0,\min\{(32\gamma)^{-1}, T\}]$  there is a map 
$$Q_{s,t}:  \Big\{ r_s \leq \frac{5}{4}\sqrt{\gamma t+s(\Lambda+1)^2}\Big\}\rightarrow N,$$  diffeomorphism onto its image, such that 
 \begin{equation*}
 \big\{ \mathbf r_s\leq \sqrt{\gamma t}  \big \} \subset \mathrm{Im}\, Q_{s,t} \subset \Big\{ \mathbf r_s\leq \frac{3}{2}\sqrt{\gamma t +s(\Lambda+1)^2}  \Big\}
 \end{equation*}
 and for any non-negative index $j\leq k$
\begin{equation*}
\big|\big((t+s)^{1/2}\nabla^{g_e(t+s)}\big)^j  \left((Q_{s,t}^{-1})^* g(t) - g_e(t+s)\right)\big|_{g_e(t+s)} < \varepsilon
\end{equation*}
in $\mathrm{Im}\, Q_{s,t}$.
\end{enumerate}
\end{theorem}

Assuming a bound on the initial curvature outside of the conical and expanding region, the above result implies a global bound for the curvature in time:

\begin{corollary}\label{corollary}
Let $(M,G)\in \mathcal M(\eta_0,\Lambda,s)$ for $0<s\leq s_0$, where $\eta_0(g_N)$, $s_0(\Lambda)$ are given by Theorem \ref{main_thm}. Suppose that $\sup_{M\setminus \mathrm{Im\Phi_s}} |\riem(G)|_{G}\leq A$. Then, there exist  $T(A,g_N)$ and $C(A,g_N)$  such that the Ricci flow $g(t)$ with $g(0)=G$ exists for $t\in [0,T]$ and satisfies
\begin{equation*}
\max_{M\times [0,T]} |\riem(g)|_g \leq \frac{C(A,g_N)}{t}.
\end{equation*}
 Moreover, all the conclusions of Theorem \ref{main_thm} hold.
\end{corollary}
\begin{proof}
Since $G$ is complete with bounded curvature, Shi's theorem \cite{Shi89} provides a complete Ricci flow $g(t)$ with bounded curvature for $t\in [0,T_s]$. Then,  the second inequality of part (1) of Theorem \ref{main_thm} implies that
\begin{equation}
|\riem(g(t))|_{g(t)}\leq C(g_N)
\end{equation}
 along the level set $\{r_s=3/4\}$  for $t\in [0, \min\{\tau_0, T_s\}]$.
 
The evolution equation for the norm of the curvature tensor along Ricci flow
\begin{equation}
\frac{\partial}{\partial t} |\riem(g(t))|_{g(t)}^2 \leq \Delta_{g(t)} |\riem(g(t))|_{g(t)}^2 + c(n)  |\riem(g(t))|_{g(t)}^3
\end{equation}
and maximum principle imply that there exists $\tau_1(A,g_N)\leq \tau_0$, such that
\begin{equation}
\max_{M\setminus \{r_s\leq 3/4\}}|\riem(g(t))|_{g(t)}\leq C(A,g_N),
\end{equation}
for $t\in [0,\min\{\tau_1,T_s\}]$, for some $C(A,g_N)$.

Since, by Theorem \ref{main_thm}, 
\begin{equation}
\max_{\{r_s\leq 3/4\}} |\riem(g(t))|_{g(t)}\leq \frac{C(g_N)}{t},
\end{equation}
for $t\in [0,\min\{\tau_0,T_s\}]$, it follows that $g(t)$ exists for all $t\in [0,\tau_1]$. This suffices to prove the result.
\end{proof}

\begin{remark}
Since the expander is merely asymptotic to the cone, in practice $\Lambda$ depends on $\eta$. Namely, one has to go far into the asymptotic region of the expander, i.e. ~make $\Lambda$ large, for the metric to be close to the cone, otherwise the class $\mathcal M(\eta,\Lambda,s)$ is empty. Thus, when we apply Corollary \ref{corollary} in Section \ref{flowing_sing}, it will be important that the statement holds for arbitrary $\Lambda$ with $\eta$ independent of $\Lambda$. 
\end{remark}

The idea behind the proof of Theorem \ref{main_thm} is that Perelman's pseudolocality theorem will control the flow in the conical region, and a localised version of the weak stability result of Deruelle--Lamm \cite{DeruelleLamm16} for expanders with positive curvature operator will control the flow in the expanding region. However, to exploit the latter we need to work with the Ricci--DeTurck flow 
\begin{eqnarray}
\frac{\partial}{\partial t} \hat g=-2\ric(\hat g) + \mathcal L_{\mathcal W(\hat g, \tilde g)} \hat g,
\end{eqnarray}
where $\mathcal W(\hat g, \tilde g)_k=\hat g_{kl}\hat g^{ij}(\hat \Gamma_{ij}^l -\tilde \Gamma_{ij}^l)$ and $\tilde g(t)$ is a carefully chosen family of background metrics defined as follows. Given $(M,\hat g(0))\in \mathcal M(\eta,\Lambda,s)$, 
\begin{eqnarray}
\tilde g(t)= \xi_1(r_s) (\Phi_s^{-1})^*(g_e(t+s))+(1-\xi_1( r_s)) \hat g(0),\label{Rdtflow}
\end{eqnarray}
where $\xi_1:[0,1]\rightarrow [0,1]$ is a fixed smooth, non-increasing function which is identically equal to $1$ in $[0,\frac{1}{2}]$ and $\xi=0$ in $[\frac{5}{8},1]$. This metric interpolates between the initial metric and the expanding metric at scale $s+t$.\\

Let $(M,g(t))_{t\in [0,T]}$ be a Ricci flow with $(M,g(0))\in \mathcal M(\eta,\Lambda,s)$  and consider the harmonic map heat flow $\psi: \{r_s \leq  \frac{3}{4}\}\times [0,T]\rightarrow \{r_s \leq  \frac{3}{4}\}$:
\begin{eqnarray}
\frac{\partial}{\partial t}\psi &=&\Delta_{g(t),\tilde g(t)}\psi,\label{hmf_eqn}\\
\psi|_{t=0}&=&id_{\{r_s\leq \frac{3}{4}\}},\label{hmf_id}\\
\psi|_{\{r_s =  \frac{3}{4}\}\times [0,T]} &=& id_{\{r_s =  \frac{3}{4}\}}, \label{hmf_bndry}
\end{eqnarray}
as in \cite{Hamilton75}, where we assume that $T$ is small enough so that both $g(t)$ and $\psi_t(\cdot):=\psi(\cdot,t)$ are smooth for $t\in [0,T]$ and $\psi_t$ is a diffeomorphism for all $t\in [0,T]$. Note that $\psi_t$ is smooth up to the corner $\{r_s = \frac{3}{4}\}\times \{0\}$, since $\tilde g(0)$ and $g(0)$ coincide around $\{r_s = \frac{3}{4}\}$. It is well known that  $$\hat g(t)=(\psi_t^{-1})^* g(t)$$ is a solution to (\ref{Rdtflow}), see \cite{ChowLuNi06}.

Lemma \ref{conical_region} below controls $\hat g(t)$ in the conical region, assuming a bound on $|\nabla \psi|_{g,\tilde g}$, and Lemma \ref{expanding_region} uses the weak stability of the expander to control $\hat g$ in the expanding region, assuming control of $\hat g$ in the overlap of the two regions.

\begin{lemma}[Estimates in the conical region]\label{conical_region}
 Given $B,\alpha>0$ there exist $\eta_2(\alpha)>0$, $\gamma_2(B,\alpha)>1$ and $C(g_c)>0$  such that the following holds: 
 
 Let $(M,g(t))_{t\in [0,T]}$ be a complete Ricci flow with bounded curvature and suppose that $(M,g(0))\in \mathcal M(\eta_2 , \Lambda, s)$, for some $\Lambda \geq \Lambda_0$ and $s\leq \tfrac1{32(\Lambda+1)^2}$. Let $\tilde g$, $\psi$  and $\hat g(t)=(\psi_t^{-1})^* g(t)$ be as above, define
 \begin{equation}
 \mathcal D^{cone}_{\gamma, \Lambda,s}=\left\{(x,t)\in \{r_s \leq  \frac{3}{4}\} \times [0,(32\gamma)^{-1}],\;  r_s(x)\geq \sqrt{\gamma t +s\Lambda^2}
 \right\},
 \end{equation}
for some $\gamma\geq\gamma_2$ and suppose $|\nabla \psi|_{g,\tilde g}\leq B$ in $\{r_s \leq  \frac{3}{4}\}\times [0,\min\{(32\gamma)^{-1},T\}]$. Then the estimates
\begin{eqnarray}
 | \hat{g}-\tilde g|_{\tilde g}+r_s|\tilde \nabla  \hat{g}|_{\tilde g}&<& \alpha,\label{conical_estimate}\\
\sum_{j=0}^2 r_s^{2+j} |(\nabla^g)^j\riem(g)|_g &\leq& C
\end{eqnarray}
are valid in $\mathcal D^{cone}_{\gamma,\Lambda+1,s}\cap (  M\times [0,T])$.
\end{lemma}

\begin{lemma}[Estimates in the expanding region]\label{expanding_region}
For every $\varepsilon>0$ and integer $k\geq 0$ there exists $\alpha_0(g_N,\varepsilon,k)>0$ such that if $(M,g(t))_{t\in [0,T]}$ is a complete Ricci flow with bounded curvature and $(M,g(0))\in\mathcal M(\alpha,\Lambda,s)$, for $\alpha\leq\alpha_0$ and some $\Lambda$ and $s<\frac{1}{32(\Lambda+1)^2}$, then the following holds: Let $\hat g(t)=(\psi^{-1}_t)^* g(t) $ be the corresponding Ricci--DeTurck flow in $\{r_s\leq \frac{3}{4}\}$. If for some $\gamma\geq 1$ estimate (\ref{conical_estimate}) holds in $\mathcal D^{cone}_{\gamma,\Lambda+1,s}\cap (M\times [0,T])$ then for every $0\leq j \leq k$:
\begin{eqnarray}
 (t+s)^{j/2}|\tilde\nabla^j (\hat g-\tilde g)|_{\tilde g} < \varepsilon,
\end{eqnarray}
in $\mathcal D^{exp}_{\gamma,\Lambda,s}\cap (M\times [0,T])$, where 
\begin{equation}
\mathcal D^{exp}_{\gamma,\Lambda,s}=\left\{ (x,t)\in M\times [0,(32\gamma)^{-1}],\; r_s(x)\leq \frac{3}{2} \sqrt{\gamma t +s(\Lambda+1)^2}
\right\}.
\end{equation}
\end{lemma}
\begin{remark}
Note that for $t\in [0,(32\gamma)^{-1}]$ and $ s \leq \frac{1}{32(\Lambda+1)^2}$ we have $$2\sqrt{\gamma t+s(\Lambda+1)^2}\leq \frac{1}{2},$$ hence $\tilde g(t)=(\Phi_s^{-1})^*g_e(t+s)$ in $\mathcal D^{exp}_{\gamma,\Lambda,s}\cap (M\times \{t\})$.
\end{remark}

Assuming for now Lemmata \ref{conical_region} and \ref{expanding_region} we proceed to prove Theorem \ref{main_thm}.
\begin{proof}[Proof of Theorem \ref{main_thm}]
Let $(M,g(t))_{t\in [0,T]}$ be a complete Ricci flow with bounded curvature such that $(M,g(0))\in \mathcal M(\eta,\Lambda,s)$, for some $\Lambda\geq \Lambda_0$ and $s\in(0, \frac{1}{32(\Lambda+1)^2}]$. We will prove that the assertion of the theorem is true when $\eta=\min\{\alpha_0, \eta_2(\alpha_0)\}$, where $\alpha_0=\alpha_0(g_N,10^{-2},4)$ is the constant provided by Lemma \ref{expanding_region} and $\eta_2(\alpha_0)$ the constant provided by applying  Lemma \ref{conical_region} for a large enough constant $B>0$, which will be specified in the course of the proof.

Let $\psi$ satisfy (\ref{hmf_eqn})-(\ref{hmf_bndry}) and define 
\begin{align*}
T_*:=\max\big\{\tau,\;
\hat g(t):=(\psi^{-1}_t)^* g(t)\textrm{ is smooth and } \big.\\
\big.|\nabla \psi|_{g,\tilde g}\leq B 
 \textrm{ in }\{r_s\leq 3/4\}\times [0,\tau)
\big\}.
\end{align*}
Applying Lemma \ref{conical_region} we obtain $\gamma_2=\gamma_2(B,\alpha_0)$ such that
\begin{eqnarray}
|\hat g-\tilde g|_{\tilde g}+r_s|\tilde \nabla \hat g|_{\tilde g}&<&\alpha_0, \label{control_conical_region_1} \\
\sum_{j=0}^2 r_s^{2+j} |(\nabla^g)^j\riem(g)|_g &\leq& c_1(g_c,A) \label{control_conical_region_2}
\end{eqnarray}
 in $\mathcal D^{cone}_{\gamma_2,\Lambda+1,s }\cap (M\times [0,T_*])$.
 
 Then, Lemma \ref{expanding_region} implies that 
 \begin{eqnarray}
 |\hat g-\tilde g|_{\tilde g}+ \sqrt{t} |\tilde\nabla \hat g|_{\tilde g}+ t |\tilde\nabla^2 \hat g|_{\tilde g} &<& 0.01,  \label{control_expanding_region}
 \end{eqnarray}
 hence
 \begin{eqnarray*}
|\riem(g)|_g &\leq& \frac{c_2(g_N)}{t}\label{control_curvature}
 \end{eqnarray*}
 in $\mathcal D^{exp}_{\gamma_2, \Lambda,s }\cap (M\times [0,T_*])$.  
 
Since $(\mathcal D^{cone}_{\gamma_2,\Lambda+1,s} \cup \mathcal D^{exp}_{\gamma_2,\Lambda,s})\cap (M\times [0,T_*])= \{r_s \leq 3/4\} \times [0,\min\{(32\gamma_2)^{-1}, T_*\}]$ and $|\nabla \psi |^2_{g,\tilde g}= \tr_{\hat g} \tilde g$, it follows from (\ref{control_conical_region_1}) and (\ref{control_expanding_region}) that 
\begin{eqnarray}
|\nabla \psi|_{g,\tilde g}\leq c_3(g_N), \label{control_diffeos}
\end{eqnarray}
 in $\{r_s \leq 3/4\}\times [0,\min\{T_*,(32\gamma_2)^{-1}\}]$.
 
 Now, choosing $B=2c_3$,  \eqref{control_curvature} implies that $g(t)$ remains smooth up to time $\min\{T_*, (32\gamma_2)^{-1}\}$. This, together with \eqref{control_diffeos} and parabolic regularity implies that $\psi_t$ is also smoothly controlled up to time $\min\{T_*, (32\gamma_2)^{-1}\}$, and remains a diffeomorphism due to \eqref{control_conical_region_1} and \eqref{control_expanding_region}. It follows that $T_*> (32\gamma_2)^{-1}$ and the estimates in the statement of the theorem are valid for $t\leq \tau_0:= (32\gamma_2)^{-1}$.
 
 In order to prove the second part of the theorem, let $(M,g(0))\in \mathcal M(\eta_1, \Lambda,s)$ for
\begin{equation*}
 \eta_1=\min\{ \alpha_0(g_N,\varepsilon,k), \eta_2(\alpha_0(g_N,\varepsilon,k))\},
 \end{equation*}
putting $B=2c_3$. Combining Lemmata \ref{conical_region} and \ref{expanding_region} as above, we obtain that, for $0\leq j\leq k$,
 \begin{equation}
  |((t+s)^{1/2}\tilde \nabla)^j (\hat g-\tilde g)|_{\tilde g} < \varepsilon , \label{3.1closeness}
 \end{equation}
in $\mathcal D^{exp}_{\gamma, \Lambda,s}$ and
 \begin{equation}
 |\hat g-\tilde g|_{\tilde g}+r_s|\tilde \nabla \hat g|_{\tilde g}<\alpha_0(g_N,\varepsilon) \label{3.1closeness2}
\end{equation}
 in $\mathcal D^{cone}_{\gamma,\Lambda-1,s}$.  
  
 Set $\tilde\tau(\gamma)=(32\gamma)^{-1}$. We claim that making $\gamma$ even larger 
 \begin{equation}
 \psi_t\Big(\Big\{ r_s \leq \frac{5}{4}\sqrt{\gamma t+s(\Lambda+1)^2}  \Big\}\Big) \subset \frac{3}{2}\Big\{ r_s \leq \sqrt{\gamma t+s(\Lambda+1)^2}  \Big\}\label{diffeo_inclusion},
 \end{equation}
 for all $t\in [0,\tilde\tau]$.

To prove this, let $t\in [0,\tilde\tau]$ and suppose there is $x\in \Big\{ r_s \leq \frac{5}{4}\sqrt{\gamma t +s(\Lambda+1)^2}\Big\}$ and $\tau_1<\tau_2<t$ such that
\begin{equation}\label{interval}\begin{split}
r_s(\psi_{\tau_1} (x)) &=\frac{5}{4}\sqrt{\gamma t +s(\Lambda+1)^2}, \\
r_s(\psi_{\tau_2}(x) ) &= \frac{3}{2}\sqrt{\gamma t+s(\Lambda+1)^2}, \quad\mathrm{and} \\
r_s(\psi_{\tau}(x)) &\in \Big[\frac{5}{4}\sqrt{\gamma t +s(\Lambda+1)^2}, \frac{3}{2}\sqrt{\gamma t+s(\Lambda+1)^2}\Big] 
\end{split}
\end{equation}
for all $\tau\in[\tau_1,\tau_2]$.

Then, for every $\tau\in [\tau_1,\tau_2]$ we have
\begin{equation}\label{3.radial_c}
\begin{aligned}
\frac{d}{d\tau}r_s(\psi_\tau(x)) &=\tilde g(\tilde \nabla r_s , \mathcal W(\hat g,\tilde g) )(\psi_\tau(x),\tau)\\
&\leq c_4 |\tilde \nabla r_s|_{\tilde g} |\tilde \nabla \hat g|_{\tilde g} (\psi_\tau (x),\tau)\\
&\leq c_4 (r_s(\psi_\tau(x) ))^{-1}  |\tilde \nabla r_s|_{\tilde g}(\psi_\tau(x),\tau),
\end{aligned}
\end{equation}
where we used \eqref{3.1closeness2}. Note that the constant $c_4$ is independent of $\gamma$ but is allowed to change from line to line.

Note that 
\begin{equation}
|\tilde\nabla r_s |_{\tilde g}(y,\tau)= | \nabla^{g_0(\tau/s)} \mathbf r |_{g_0(\tau/s)} (\varphi_s(\Phi_s^{-1}(y)))\leq 2,\label{3.grad_est}
\end{equation}
as long as $r_s(y)\geq \sqrt{\gamma_0 \tau+s\Lambda_0^2}$, by Lemma \ref{exp_con_est}.

Since $t>\tau_2$, it follows from \eqref{interval} that, for $\tau \in [\tau_1,\tau_2]$,
\begin{equation*}
r_s(\psi_\tau(x))\geq \frac{5}{4}\sqrt{\gamma t+s(\Lambda+1)^2}>\sqrt{\gamma_0 \tau+s\Lambda_0^2},
\end{equation*}
as long as $\gamma\geq\gamma_0$ and $\Lambda\geq \Lambda_0$. Hence \eqref{3.grad_est} holds at $(\psi_\tau(x),\tau)$.

Putting this into \eqref{3.radial_c} we obtain
\begin{equation}
\frac{d}{d\tau} r_s(\Psi_\tau(x)) \leq c_4 (\gamma t)^{-1/2},\label{ODE}
\end{equation}
for $\tau\in [\tau_1,\tau_2]$. Integrating this we obtain
\begin{align*}
\frac{1}{4} \sqrt{\gamma t + s(\Lambda+1)^2} &<c_4 \left(\frac{t}{\gamma}\right)^{1/2}
\end{align*}
If $\gamma\geq 4 c_4$  we obtain a contradiction. Hence $\tau_2\geq t$, which implies that \eqref{diffeo_inclusion} holds for every $t \in [0,\tilde \tau]$.

Similarly we obtain the inclusion
\begin{equation*}
\big\{r_s \leq  \sqrt{\gamma t+s(\Lambda+1)^2} \big\} \subset  \psi_t \Big( \Big\{ r_s \leq \frac{5}{4}\sqrt{\gamma t + s(\Lambda+1)^2}  \Big\}\Big),
\end{equation*}

The conclusion of the theorem then holds for $Q_{s,t}=\Phi_s^{-1} \circ \psi_t$.
\end{proof}

\section{Proofs of Lemmata \ref{conical_region} and \ref{expanding_region}}\label{main_estimates}
\subsection{Estimates in the conical region}
First we need the following auxiliary lemma. Let $(M,g)$ be a complete Riemannian manifold with boundary. For every $x\in M\setminus \partial M$, recall that the $C^{2,\alpha}$-harmonic radius $r_{har,g}(x)$ at $x$ is the maximal $r<d_g(x,\partial M)/2$ with the following property: there exist harmonic coordinates $u: B_g(x,r)\rightarrow \mathbb R^n$ satisfying $u(x)=0$ and
\begin{equation}\label{harmonic_coords}
\begin{aligned}
&2^{-1 }\delta\leq g \leq 2 \delta,\\
\sum_{i,j,k}r|\partial_k g_{ij}|_{C^0}+&\sum_{i,j,k,l} r^2( |\partial^2_{kl}g_{ij}|_{C^0}+
r^{\alpha} [\partial^2_{kl} g_{ij}]_\alpha )\leq 2,
\end{aligned}
\end{equation}
 where $\delta$ here denotes the Euclidean metric in $\mathbb R^n$.

If $x\in \partial M$, the harmonic radius $r_{har, g}(x)$ is defined as the maximal $r$ such that there exists $u: B_g(x,r)\rightarrow \mathbb R^n$, mapping $B_g(x,r)$ to $\{x^n \geq 0\}$ and $B_g(x,r)\cap \partial M$ to $\{x^n=0\}$, such that  \eqref{harmonic_coords} holds and  the restriction $u|_{\partial M}$ is harmonic (see \cite{Anderson04}).

The following lemma proves a `pseudolocality'-theorem for the harmonic map heat flow.  We would like to stress that this is not a true pseudolocality-theorem since it assumes an a-priori bound \eqref{map_control} on the gradient of the solution to the harmonic map heat flow with respect to the evolving metrics. Nevertheless, in the application later we will be able to assume such a bound, and then show a-posteriori that this bound is never achieved.  Notably, due to the assumed bound on the gradient, the proof relies only on parabolic regularity. 

\begin{lemma}\label{hmf_pseudolocality}
For every $\alpha,B>0$ there is an $\varepsilon_h=\varepsilon_h(\alpha,B)>0$ with the following property.  Let $g(t),\tilde g(t)$, $t\in [0,T]$ be one-parameter families of Riemannian metrics on a smooth manifold $M^n$ with boundary $\partial M$ and that $g(0)=\tilde g(0)$ in a neighbourhood of $\partial M$. Also, let $\psi: M\times [0,T]\rightarrow M$ be a solution to the harmonic map flow
\begin{eqnarray*}
\frac{\partial}{\partial t}\psi&=&\Delta_{g,\tilde g} \psi,\\
\psi|_{t=0}&=& id_M,\\
\psi|_{\partial M\times[0,T]}&=& id_{\partial M}.
\end{eqnarray*}
Suppose that $r_{har,g(0)}(x)> \rho$ for some $x\in M$ and
\begin{gather}
|\nabla \psi|_{g,\tilde g}\leq B, \label{map_control}\\
\sum_{j=0}^2 \rho^{2+j} \left( \left| \frac{\partial}{\partial t} (\nabla^{g(0)})^j g\right|_{g(0)}+\left| \frac{\partial}{\partial t} (\nabla^{g(0)})^j \tilde g  \right|_{g(0)}   \right) \leq B \label{tderiv_control}
\end{gather}
in $B_{g(0)}(x,\rho)\times [0,\min\{\varepsilon_h\rho^2,T\}]$ and 
\begin{equation}
B^{-1}g(0) \leq \tilde g(0)\leq B g(0),\quad \sum_{j=1}^2 \rho^j | (\nabla^{g(0)})^j \tilde g(0)|_{g(0)} \leq B   \label{initial_control_tilde}
\end{equation}
at  $B_{g(0)}(x,\rho)$. Then, $\psi_t(\cdot)|_{B_{g(0)}(x, \rho/10)}:=\psi(\cdot,t)$ is a diffeomorphism onto its image for every $t\in[0,\min\{\varepsilon_h \rho^2,T\}]$ and
\begin{eqnarray*}
|(\psi^{-1})^* g - g|_{g(0)}+\rho |  \nabla^{g(0)} ((\psi^{-1})^* g - g)|_{g(0)} < \alpha ,
\end{eqnarray*}
in $B_{g(0)}(x, \frac{\rho}{10})\times [0, \min\{\varepsilon_h \rho^2,T\}]$.
\end{lemma}

\begin{proof}
By rescaling $g'(t)=\rho^{-2}g(\rho^2 t)$,  $\tilde g'(t)=\rho^{-2}\tilde g(\rho^2 t)$ and $\psi'(\cdot,t)=\psi(\cdot, \rho^2 t)$ we may assume that $\rho=1$.

First suppose that $x\not\in \partial M$. In harmonic coordinates $u$ in the ball $B_{g(0)}(x,1)$ we may write $u\circ\psi\circ u^{-1}=(\psi^1,\ldots,\psi^n)$. Then,
\begin{eqnarray}
\frac{\partial \psi^l}{\partial t}  &=& g^{ij}\frac{\partial^2 \psi^l}{\partial x^i \partial x^j}  - g^{ij}\Gamma^k_{ij} \frac{\partial \psi^l}{\partial x^k} + g^{ij}(\tilde \Gamma^l_{mk}\circ \psi)\frac{\partial \psi^m}{\partial x^i} \frac{\partial \psi^k}{\partial x^j}, \label{heat_coords1}\\
\psi^l|_{t=0}&=&x^l. \label{heat_coords2}
\end{eqnarray}

Observe that by \eqref{map_control} there exists $\varepsilon_B>0$ such that if $u\circ\psi_t\circ u^{-1}(B_{1/8 -\varepsilon_B}) \subset B_{1/4}$ then $u\circ\psi_t\circ u^{-1}(B_{1/8 })\subset B_{1/2}$.

By continuity, there exists a maximal $\tau\in(0,\min\{1,T\} ]$, such that $u\circ \psi_t \circ u^{-1} (B_{1/8}) \subset B_{1/2}$, for every $t\in [0,\tau]$. Hence, $\psi^l$ are controlled in $L^p(B_{1/8}\times [0,\tau] )$, for $p>n+2$. The assumptions of the lemma imply that the last term in (\ref{heat_coords1}),
 $$g^{ij}\big(\tilde \Gamma^l_{mk}\circ \psi\big)\frac{\partial \psi^m}{\partial x^i} \frac{\partial \psi^k}{\partial x^j},$$ 
 is also uniformly controlled in $L^p(B_{1/8}\times [0,\tau] )$. 
 
Parabolic regularity then implies that $\psi^l$ are controlled in $W^{2,1}_p(B_{1/8-\varepsilon_B}\times [0,\tau])$. By the embedding of $W^{2,1}_p\subset C^{1+\zeta,(1+\zeta)/2}$ for $\zeta=1-\frac{n+2}{2}$, and parabolic regularity again, it follows that 
\begin{equation}
|\psi^l |_{2+\zeta, (2+\zeta)/2} \leq C(B) \label{c2a_est}
\end{equation}
 in $B_{1/8-\varepsilon_B}\times [0, \tau]$. 

Now, observe that there exists $\tau_B\in(0,\min\{1,T\} ]$, depending only on $C(B)$, such that if \eqref{c2a_est} holds in  $B_{1/8-\varepsilon_B}\times [0, \tau_B]$ then $u\circ\psi_t \circ u^{-1}(B_{1/8-\varepsilon_B})\subset B_{1/4}$ for $t\in [0,\tau_B]$. From the above this gives $u\circ\psi_t \circ u^{-1}(B_{1/8})\subset B_{1/2}$ for all $t\in [0,\tau_B]$.Hence,  $\tau\geq \tau_B$.

Finally, it follows from \eqref{c2a_est} that for every $\alpha>0$ there is $\varepsilon_h=\varepsilon_h(\alpha,B)$ small enough so that for all $i,j,l$
\begin{eqnarray*}
\left|\psi^l -x^l \right|+ \left|\frac{\partial \psi^l}{\partial x^i} - \delta_{li}\right| + \left|\frac{\partial^2 \psi^l}{\partial x^i \partial x^j}\right| < \alpha,
\end{eqnarray*}
in $B_{1/8-\varepsilon_B}\times [0, \varepsilon_h]$, which suffices to prove the result, if $\varepsilon_B$ is chosen small enough.

If $x\in\partial M$, in addition to \eqref{heat_coords1}-\eqref{heat_coords2} holding in $B_{1/8}\cap \{x^n\geq 0\}$, we also have the following boundary conditions on $B_{1/8}\cap \{x^n =  0\}$:
\begin{equation}\label{heat_boundary}
\begin{aligned}
\psi^l|_{\{x^n=0\}} &= x^l,\;1\leq l \leq n-1,\\
\psi^n |_{\{x^n=0\}} &= 0.
\end{aligned}
\end{equation}
Since $g(0)=\tilde g(0)$ in a neighbourhood of $\partial M$ and $\psi|_{t=0}=id_M$ it follows that the compatibility conditions required for the $C^{2+\zeta,(2+\zeta)/2}$ estimates hold. The result then follows arguing as in the interior case.
\end{proof}

\begin{proof}[Proof of Lemma \ref{conical_region}]
We first recall a direct consequence of Perelman's pseudolocality theorem and Shi's local derivative estimates, \cite[Corollary A.5]{Topping10}). 

There exists $\varepsilon_{ps}>0$  depending only on $n$, such that the following holds: Let $(g(t))_{t\in [0,T]}$ be a complete, bounded curvature Ricci flow on an $n$-dimensional manifold $M$. Assume that, for some $r>0$ and $x_0 \in M$,
\begin{gather}
 \sum_{j=0}^2 r^j\big|(\nabla^{g(0)})^j\riem_{g(0)}\big|_{g(0)} \leq r^{-2}, \text{ in } B_{g(0)}(x_0,r), \label{pseudolocality.1}\\
 \vol_{g(0)}(B_{g(0)}(x_0,r)) \geq (1-\varepsilon_{ps}) \omega_n r^n,\label{pseudolocality.2}
\end{gather}
then
\begin{equation} \sum_{j=0}^2 r^j |(\nabla^g)^j \riem|_{g}(x,t) \leq (\varepsilon_{ps}r)^{-2}, \label{curvature_estimate_0}
\end{equation}
for $t\in [0, \min \{T, (\varepsilon_{ps}\rho(x))^2 \} ]$ and $x \in B_{g(0)}(x_0,\varepsilon_{ps} r)$.
 
Let $(M,g(0))\in \mathcal M(\eta,\Lambda,s)$. For sufficiently small $\eta$ we can choose $\beta, c_0 >0$, depending only on $g_c$  such that the following holds: Let $\rho(x)=\beta  r_s(x)$. Then  for all $x \in \{r_s\leq 3/4\}$ the condition \eqref{pseudolocality.1} is fulfilled with $r = \rho(x)$. Furthermore,  
for $x\in \{ (\Lambda+1)\sqrt s \leq r_s \leq 3/4 \}$ condition \eqref{pseudolocality.2} is fulfilled with $r = \rho(x)$. 

Moreover, if $x\in\{r_s=3/4\}$
\begin{equation*}
r_{har,g(0)}(x)\geq c_0,
\end{equation*}
and
\begin{equation}\label{harmonic_radius}
r_{har,g(0)}(x)\geq c_0\rho (x),
\end{equation}
for $x\in \{  (\Lambda+1)\sqrt s \leq r_s\leq 3/4 - c_0/2\}$, by the lower semicontinuity of the harmonic radius.

Then by \eqref{curvature_estimate_0}, for all $x\in \{ (\Lambda+1)\sqrt s  \leq r_s \leq 3/4\}$, 
\begin{equation}
\sum_{j=0}^2 (\rho(x))^j |(\nabla^g)^j \riem|_{g}(x,t) \leq (\varepsilon_{ps}\rho(x))^{-2}, \label{curvature_estimate}
\end{equation} 
for $t\in [0, \min \{T, (\varepsilon_{ps}\rho(x))^2 \} ]$. Now, using (\ref{curvature_estimate}) and integrating the Ricci flow equation we estimate
\begin{equation*}
\sum_{j=0}^2 \left(\rho^{2+j} \left| \frac{\partial}{\partial t}(\nabla^{g(0)})^j (g - g(0))\right|_{g(0)}\right) (x,t) \leq C(n),
\end{equation*}
and
\begin{equation}
\sum_{j=0}^2 \left(\rho^j \left| (\nabla^{g(0)})^j (g- g(0))\right|_{g(0)}\right)(x,t) \leq C(n) \frac{t}{(\rho(x))^2}, \label{g_close}
\end{equation}
for $x\in \{ (\Lambda+1)\sqrt s \leq r_s \leq 3/4 \}$ and $t\in [0, \min \{T, (\varepsilon_{ps}\rho(x))^2 \} ]$. 

Similarly, since
\begin{gather*}
(1-2\eta)g(0)\leq \tilde g(0) \leq (1+2\eta) g(0), \\
 r_s\big|\nabla^{g(0)} \tilde g(0)\big|_{g(0)}+r_s^2 \big|(\nabla^{g(0)})^2 \tilde g(0)\big|_{g(0)}\leq C(\xi_1, g_c)\eta
\end{gather*}
on $\{ (\Lambda+1)\sqrt s \leq r_s \leq 3/4 \}$, by the pseudolocality theorem applied to $(N,g_e(s+t))_{t\geq 0}$, we obtain
\begin{eqnarray*}
\sum_{j=0}^2 \left(\rho^{2+j}  \left| \frac{\partial}{\partial t}(\nabla^{g(0)})^j (\tilde g-g(0))\right|_{g(0)}\right) (x,t) \leq C(\xi_1,g_c),
\end{eqnarray*}
for $x\in \{ (\Lambda+1)\sqrt s \leq r_s \leq 3/4 \}$ and $t\in [0, \min \{T, (\varepsilon_{ps}\rho(x))^2 \} ]$. Integrating the Ricci flow equation leads to
\begin{equation}
\sum_{j=0}^2 \left(\rho^{j} \left| (\nabla^{g(0)})^j (\tilde g- g(0))\right|_{g(0)}\right)(x,t) \leq C(\xi_1,g_c)\frac{t}{(\rho(x))^2}, \label{tilde_close}
\end{equation}
for $x\in\{ (\Lambda+1)\sqrt s \leq r_s \leq 3/4 \}$ and $t\in [0, \min \{T, (\varepsilon_{ps}\rho(x))^2 \} ]$.

Hence, by  Lemma \ref{hmf_pseudolocality}, for every $\varepsilon>0$ there is $\gamma(g_c,B,\beta,\varepsilon)>1$ large enough such that
\begin{equation*}
 \Big(\big| (\Psi^{-1})^*g -g\big|_{g(0)}+ r_s\big|\nabla^{g(0)}((\Psi^{-1})^* g-g)\big|_{g(0)}\Big)(x,t)< \varepsilon,
\end{equation*}
for $x\in \{ (\Lambda+1)\sqrt s \leq r_s \leq 3/4\}$ and $t\in [0,\min\{T,\gamma^{-1}(r_s  (x))^2\}]$. Then, at any such $(x,t)$ we may estimate, by possibly making $\gamma$ even  larger (exploiting (\ref{g_close}) and (\ref{tilde_close})) and $\eta$ smaller
\begin{equation*}
\big|(\psi^{-1})^* g -\tilde g\big|_{\tilde g}\leq 2\,\Big(\big|(\psi^{-1})^* g-g\big|_{g(0)}+ \big|g-g(0)\big|_{g(0)}
 +|\tilde g - g(0)|_{g(0)}\Big) < \alpha
\end{equation*}
and
\begin{flalign*}
\big|\tilde \nabla (\psi^{-1}&)^* g \big|_{\tilde g} \leq 2\left(\big|\tilde \nabla ((\psi^{-1})^*g-g)\big|_{g(0)} + \big|\tilde \nabla (g-\tilde g)\big|_{g(0)}\right),\\
&\leq 2 \left(\big|(\tilde \nabla - \nabla^{g(0)})((\psi^{-1})^* g- g)\big|_{g(0)}+ \big|\nabla^{g(0)}((\psi^{-1})^* g - g)\big|_{g(0)} \right. \\
 & \left.\qquad\ \ + \big|(\tilde\nabla-\nabla^{g(0)})(g-\tilde g)\big|_{g(0)}+\big|\nabla^{g(0)}(g-\tilde g)\big|_{g(0)} \right)\\
&\leq C\, \big|\nabla^{g(0)} \tilde g \big |_{g(0)} \left( \big|(\psi^{-1})^* g-g \big|_{g(0)} + \big|g-g(0)\big|_{g(0)}+\big|\tilde g -g(0)\big|_{g(0)}  \right) \\
 & \ \ \ \ +C\Big( \big|\nabla^{g(0)} ((\psi^{-1})^* g-g)\big|_{g(0)} + \big|\nabla^{g(0)} (g-g(0))\big|_{g(0)} \\
 & \qquad\ \ \ \ \ \ + \big|\nabla^{g(0)} (\tilde g-g(0))\big|_{g(0)}\Big) <\frac{\alpha}{ r_s},
\end{flalign*}
which suffices to prove the theorem.
\end{proof}

\begin{proof}[Proof of Lemma \ref{exp_con_est}]
Choosing $\Lambda_0$ large, $|F^* g_0(0) - g_c|_{g_c}+\mathbf r |\nabla^{g_c} F^* g_0(0)|_{g_c}$ becomes small enough in $\{\mathbf r \geq \Lambda_0\}$ so that
\begin{equation*}
2/3\leq|\nabla^{g_0(0)} \mathbf r |_{g_0(0)}\leq 3/2\quad\mathrm{and}\quad
|\mathbf r \Delta_{g_0(0)} \mathbf r |\leq  2(n-1),
\end{equation*}
since $F^*\mathbf r =r$, $|\nabla^{g_c} r|_{g_c}=1$ and the mean curvature of the level sets of $r$ is $\Delta_{g_c} r= (n-1)/r$. Moreover, by the quadratic curvature decay we obtain 
\begin{equation*}
\mathbf r^2 |\riem(g_0(0))|_{g_0(0)} \leq C(g_c)/2.
\end{equation*}
Then, using Perelman's pseudolocality theorem as in the proof of Lemma \ref{conical_region}, we obtain the result.
\end{proof}

\subsection{Estimates in the expanding region}
In this section we show that we can adapt the estimates in \cite{DeruelleLamm16} to show that control in the conical region yields control in the expanding region.

\begin{lemma}\label{long_time_estimate}
Let $(N,g_N,f)$ be an asymptotically conical gradient Ricci expander with positive curvature operator and let $(g_0(t))_{t\geq 0}$ be the induced Ricci flow with $g_0(0)=g_N$. There exists $\alpha_0(g_N)>0$ with the following property: Let $\Lambda\geq \Lambda_0, \gamma\geq 1$ and $\mathbf r(x):=2\sqrt{f(x)}$. Define the interior region 
$$D=\left\{  (x,t)\in N\times [0,T],\; \mathbf r(x)\leq 2 \sqrt{\gamma t+(\Lambda+1)^2} \right\}$$
and the annular region
$$A=\left\{  (x,t)\in N\times [0,T],\; \sqrt{\gamma t+(\Lambda+1)^2} \leq \mathbf r(x)\leq 2 \sqrt{\gamma t+(\Lambda+1)^2}  \right\}\ .$$
Let $(g(t))_{t\in [0,T]}$ be a solution to the Ricci-DeTurck flow
\begin{equation*}
\frac{\partial}{\partial t}g(t)=-2\ric(g(t))+\mathcal L_{\mathcal W(g(t),g_0(t))}  g(t)
\end{equation*}
on $D$, and assume 
$$H:=\max\bigg\{
\sup_{D\cap \{t=0\}} |g-g_0|_{g_0},
\sup_A\left(|g-g_0|_{g_0}+\mathbf r |\nabla^{g_0} g|_{g_0}\right)
\bigg\}\leq\alpha_0.$$ 
If $D'=D\cap\{\mathbf r(x)\leq \frac{3}{2}\sqrt{\gamma t+(\Lambda+1)^2} \}$, then
\begin{equation*}
\sup_{D'}\big| (t^{\frac{1}{2}}\nabla^{g_0})^a (t\partial_t)^b(g-g_0)\big|_{g_0} \leq C_{a,b}(g_N),
\end{equation*}
for any non-negative indices $a,b$. Furthermore, for every $k = 0,1,\ldots$, there exists $C_k'=C_k'(g_N)$ and $0<\alpha_k(g_N)\leq\alpha_0$ such that if $H \leq \alpha_k$, then 
\begin{equation*}
\sup_{D'}\big| (t^{\frac{1}{2}}\nabla^{g_0})^a (t\partial_t)^b(g-g_0)\big|_{g_0} \leq C_{k}' H \, ,
\end{equation*}
provided $a+2b\leq k$.
\end{lemma}

\begin{proof}
Fix a smooth function $0\leq\xi_2\leq 1$, identically equal to $1$ in $[0,1]$ and $0$ in $[2,+\infty)$ and let $C_{\xi_2}>0$ be a constant such that
\begin{equation*}
|\xi_2'|+|\xi_2''| \leq C_{\xi_2}.
\end{equation*} 
Define the following cut-off function in $D$:
\begin{equation*}
\chi(x,t)=\xi_2\Big({\mathbf r(x)}\big(\gamma t +(\Lambda+1)^2\big)^{-1/2}\Big).
\end{equation*}
Since $\mathbf r > \Lambda_0$ in $A$ it follows from Lemma \ref{exp_con_est} that $|\nabla^{g_0} \mathbf r|_{g_0}^2  \leq 2$ and $|\mathbf r\Delta_{g_0} \mathbf r| \leq 4(n-1)$ in $A$. Hence, we compute
\begin{equation*}
|\nabla^{g_0(t)} \chi|^2_{g_0(t)} \leq \frac{C_{\xi_2}^2}{ \gamma t + (\Lambda+1)^2} |\nabla^{g_0(t)} \mathbf r |_{g_0(t)}^2 \leq \frac{ C_1(\xi_2)}{t+\gamma^{-1}(\Lambda+1)^2}. 
\end{equation*}
Moreover, we compute
\begin{equation*}
\partial_t \chi =-\frac{1}{2} \xi_2'\Big({\mathbf r(x)}\big(\gamma t +(\Lambda+1)^2\big)^{-1/2}\Big) \frac{\mathbf r(x)}{\sqrt{\gamma t +(\Lambda+1)^2} } \frac{1}{ t+\gamma^{-1}(\Lambda+1)^2}, 
\end{equation*}
hence $|\partial_t \chi|\leq C_{\xi_2} \big(t+\gamma^{-1}(\Lambda+1)^2\big)^{-1}$, because $\xi_2'=0$ in the set $\{\mathbf r \geq 2 \sqrt{\gamma t +(\Lambda+1)^2}\}$.

Similarly, we compute
\begin{align*}
\Delta_{g_0(t)} \chi =&\ \xi_2'\Big({\mathbf r(x)}\big(\gamma t +(\Lambda+1)^2\big)^{-1/2}\Big) \frac{1}{\sqrt{\gamma t +  (\Lambda+1)^2}} \Delta_{g_0(t)} \mathbf r \\
&+ \frac{1}{\gamma t + (\Lambda+1)^2} \xi_2''\Big({\mathbf r(x)}\big(\gamma t +(\Lambda+1)^2\big)^{-1/2}\Big)|\nabla^{g_0(t)}\mathbf r|_{g_0(t)}^2,
\end{align*}
hence $|\Delta_{g_0(t)} \chi| \leq C_2(n,\xi_2)\big(t+\gamma^{-1}(\Lambda+1)^2\big)^{-1}$, because $\xi_2'=0$  in $\{\mathbf r\leq \sqrt{\gamma t+(\Lambda+1)^2}\}$. Putting everything together gives
\begin{equation}
\big|\nabla^{g_0(t)} \chi \big|_{g_0(t)}^2 + |\partial_t \chi| + |\Delta_{g_0} \chi |\leq \frac{C_3(n,\xi_2)}{t+\gamma^{-1}(\Lambda +1)^2}, \label{cutoffest}
\end{equation}
in $D$. Moreover, since $\mathbf r(x)\geq  \big(t+\gamma^{-1}(\Lambda +1)^2\big)^{1/2}$ in $A$, we obtain
\begin{equation}
\big|\nabla^{g_0(t)} g(t)\big|_{g_0(t)} \leq \frac{H}{\sqrt{t+\gamma^{-1}(\Lambda +1)^2}} \label{dg_est}
\end{equation}
in $A$. Now, letting $h(t)=g(t)-g_0(t)$, the Ricci-DeTurck flow in $D$ takes the form
\begin{equation*}
(\partial_t-L_t) h=R_0[h]+\nabla R_1[h],
\end{equation*}
where 
\begin{align*}
L_t h_{ij}&=\Delta_{g_0(t)} h_{ij} +2\riem(g_0(t))_{iklj}h_{kl}-\ric(g_0(t))_{ik}h_{kj}-\ric(g_0(t))_{jk}h_{ki},\\
R_0[h]&=\riem(g_0(t))\ast h\ast h +O(h^3)\ast \riem(g_0(t))\\
&\quad+ g^{-1}\!\ast g^{-1}\! \ast\! \nabla^{g_0(t)} h\ast \nabla^{g_0(t)} h, \\
\nabla R_1[h]&= \nabla^{g_0(t)}_p\Big( \big( (g_0(t)+h(t))^{pq} -(g_0(t))^{pq}  \big) \nabla^{g_0(t)}_q h\Big),
\end{align*}
and $O(h^3)$ satisfies $|O(h^3)|_{g_0(t)}\leq C|h(t)|^3_{g_0(t)}$.  Also we denote 
\begin{equation*}
R_1[h]=\big( (g_0(t)+h(t))^{pq} -(g_0(t))^{pq}  \big) \nabla^{g_0(t)}_q h.
\end{equation*}
A direct computation yields the following evolution equation for  $\chi^2 h$:
\begin{equation}
\begin{split}\label{a}
(\partial_t-L_t)(\chi^2 h)=&\ \chi^2 R_0[h] + \nabla^{g_0(t)}(\chi^2 R_1[h]) \\
&+ \big(2\chi \,\partial_t\chi - 2\chi \,\Delta_{g_0(t)} \chi - 2\big|\nabla^{g_0(t)} \chi \big|^2\big) h  \\
&-2\chi\, \nabla^{g_0(t)} \chi \ast \nabla^{g_0(t)} h - 2\chi\,\nabla^{g_0(t)}\chi \ast R_1 [h] . 
\end{split}
\end{equation}

Define
\begin{align*}
P(x,R)&=\left\{ (y,t)\in N\times [0,+\infty), y\in B_{ g_0(t)}(x,R), t\in \big[0,R^2\big]    \right\},\\
Q(x,R)&=\left\{ (y,t)\in N\times [0,+\infty), y\in B_{  g_0(t) }(x,R), t\in \big[R^2/2,R^2\big]    \right\}.
\end{align*}
Given $0<T'<T$, we consider the Banach spaces $X_{T'}$ and $Y_{T'}=Y^0_{T'}+\nabla Y^1_{T'}$, with norms defined as follows, as in \cite{ DeruelleLamm16, KochLamm12}:
\begin{align*}
||h||_{X_{T'}}&= \sup_{N\times [0,T']} |h|_{g_0 }+\\
& \sup_{(x,R)\in N\times (0,\sqrt{T'})} \left( R^{-\frac{n}{2}} || \nabla^{ g_0 } h||_{L^2(P(x,R))} + R^{\frac{2}{n+4}} ||\nabla^{ g_0 } h||_{L^{n+4}(Q(x,R))} \right), \nonumber \\
||h||_{Y^0_{T'}}&=\sup_{(x,R)\in N\times (0,\sqrt{T'})} \left(  R^{-n} || h ||_{L^1(P(x,R))} + R^{\frac{4}{n+4}} ||  h||_{L^{\frac{n+4}{2}}(Q(x,R))}   \right), \\
||h||_{Y^1_{T'}}&=\sup_{(x,R)\in N\times (0,\sqrt{T'})} \left(  R^{-\frac{n}{2}} ||h||_{L^2(P(x,R))} +R^{\frac{2}{n+4}} ||h||_{L^{n+4}(Q(x,R))}  \right).
\end{align*}

Let
\begin{align*}
S_1[h]&= \chi^2 R_0[h] + \nabla^{g_0(t)}(\chi^2 R_1[h]), \\
S_2[h]&= \big(2\chi\, \partial_t\chi - 2\chi\, \Delta_{g_0(t)} \chi - 2\big|\nabla^{g_0(t)} \chi \big|^2  \big) h \\ 
&\ \ \ \ -2\chi\, \nabla^{g_0(t)} \chi \ast \nabla^{g_0(t)} h - 2\chi\, \nabla^{g_0(t)}\chi \ast R_1 [h], 
\end{align*}
as they appear in (\ref{a}). 

By (\ref{cutoffest}) and (\ref{dg_est}) it follows that $S_2[h]$ is supported in $A$ and satisfies 
\begin{equation*}
|S_2[h] |_{g_0(t)} \leq \frac{C_4 H}{t+ \gamma^{-1}(\Lambda+1)^2},
\end{equation*}
 hence, applying Lemma \ref{annulus_estimate},  we obtain
\begin{equation} \label{Est_S2}
|| S_2[h]  ||_{Y_{T'}}=|| S_2[h] ||_{Y^0_{T'}}\leq C(g_N) C_4 H.
\end{equation}

To estimate $S_1[h]$ we may estimate for the first two terms in $\chi^2 R_0[h]$:
\begin{equation}\label{S1_term0}\begin{split}
\big|\chi^2 ( h\ast h +O(h^3))\ast \riem&\big|_{g_0(t)} \leq C\chi^2 |h|_{g_0(t)}^2 |\riem(g_0(t))|_{g_0(t)},\\
&\leq C |\chi^2 h|_{g_0(t)}^2 |\riem(g_0(t))|_{g_0(t)} \\&\quad+C \chi^2 (1-\chi^2) |h|_{g_0(t)}^2  |\riem(g_0(t))|_{g_0(t)},\\
&\leq C |\chi^2 h|_{g_0(t)}^2 |\riem(g_0(t))|_{g_0(t)} \\&\quad+ \frac{C(g_N) H \chi^2(1-\chi^2)}{t+\gamma^{-1}(\Lambda+1)^2},
\end{split}
\end{equation}
since from Lemma \ref{exp_con_est}  
\begin{equation*}
|\riem(g_0)|_{g_0}\leq \frac{C(g_N)}{\mathbf r^2}\leq \frac{C(g_N)}{\gamma t+ (\Lambda+1)^2},
\end{equation*}
 in $A$.

For the term involving $\nabla^{g_0(t)} h$ we compute
\begin{equation*}\begin{split}
\chi^2 g^{-1}\ast g^{-1}\ast \nabla^{g_0(t)}h&\ast  \nabla^{g_0(t)}h\\
&= \chi^2(1-\chi^2) g^{-1}\ast g^{-1}\ast \nabla^{g_0(t)}h\ast  \nabla^{g_0(t)}h\\
&\quad+ g^{-1}\ast g^{-1}\ast \nabla^{g_0(t)}(\chi^2 h)\ast  \nabla^{g_0(t)}(\chi^2 h)\\
&\quad+g^{-1}\ast g^{-1}\ast \chi^2\ast \nabla^{g_0(t)}\chi \ast \nabla^{g_0(t)}\chi \ast h \ast h\\
&\quad+g^{-1}\ast g^{-1}\ast \chi^3 \ast \nabla^{g_0(t)}\chi \ast \nabla^{g_0(t)} h \ast h.
\end{split}
\end{equation*}
From this we may estimate
\begin{equation}\label{S1_term1}
\begin{split}
\big|\chi^2 g^{-1}\ast g^{-1}&\ast \nabla^{g_0(t)}h\ast  \nabla^{g_0(t)}h\big|_{g_0(t)}\leq C| \nabla^{g_0(t)}(\chi^2 h)|_{g_0(t)}^2  \\
&+ \chi^2(1-\chi^2) | \nabla^{g_0(t)} h |_{g_0(t)}^2 + C\chi^2|\nabla^{g_0(t)}\chi |_{g_0(t)}^2 | h |_{g_0(t)}^2 
\\ &+C \chi^3 | \nabla^{g_0(t)}\chi |_{g_0(t)} | \nabla^{g_0(t)} h |_{g_0(t)}| h|_{g_0(t)}.
\end{split}
\end{equation}
Note that the terms in the second and third line are supported in $A$ and due to \eqref{cutoffest} and \eqref{dg_est} are bounded by $CH/(t+\gamma^{-1}(\Lambda +1)^2)$. Here we assumed that w.l.o.g. $H\leq 1$.
Finally, for $\chi^2 R_1[h]$  we have
\begin{equation}\label{S1_term2}\begin{split}
|\chi^2 R_1[h]|_{g_0(t)} &\leq C |\chi^2 h|_{g_0(t)} |\nabla^{g_0(t)} (\chi^2 h)|_{g_0(t)}+C |\nabla^{g_0(t)} \chi |_{g_0(t)}|h|^2_{g_0(t)} \\
&\qquad+ C(1-\chi^2)\chi^2 |h|_{g_0(t)} |\nabla^{g_0(t)} h|_{g_0(t)},
\end{split}\end{equation}
where again the last two terms are supported in $A$ and due to \eqref{cutoffest} and \eqref{dg_est} are bounded by $CH^2/(t+\gamma^{-1}(\Lambda +1)^2)^{1/2}$.
Thus combining \eqref{S1_term0}, \eqref{S1_term1} and \eqref{S1_term2} and using Lemma \ref{annulus_estimate}, together with the estimate from Lemma 3.1 in \cite{DeruelleLamm16},  we can estimate
\begin{equation*}
||S_1[h]||_{Y_{T'}} \leq C (||\chi^2 h||_{X_{T'}}^2 +  H)\ .
\end{equation*}
We can use this estimate, together with \eqref{Est_S2}, to apply the main estimate, Theorem 6.1 in the stability result of Deruelle-Lamm, \cite{DeruelleLamm16},
to obtain
\begin{equation*}
||\chi^2 h ||_{X_{T'}} \leq C( ||\chi^2 h||_{X_{T'}}^2 +  H  ).
\end{equation*}
Therefore, for every $T'\leq T$ such that $||\chi^2 h ||_{X_{T'}} \leq \frac{1}{2C}$ we have
\begin{equation*}
||\chi^2 h ||_{X_{T'}} \leq C  H  .
\end{equation*}
Thus, if $ \max\{H,C  H\}   < 1/(2C)$ it follows that 
\begin{equation*}
||\chi^2 h ||_{X_{T}} \leq C  H  ,
\end{equation*}
since $$\lim_{T'\rightarrow 0}( ||\chi^2 h ||_{X_{T'}} - \sup_{N\times [0,T']} |\chi^2h|_{g_0})=0 \quad\mathrm{and}\quad \lim_{T'\rightarrow 0} \sup_{N\times [0,T']} |h|_{g_0}\leq H.$$

\noindent The decay estimates follow by a local argument and scaling. We split them in several steps.\\[1ex]
{\bf Claim 1:} {\it There exists $0<r_0<1$, $\varepsilon_0 > 0$ and constants $C_{a,b}>0$ such that the following holds:
Let $x_0 \in N, t_0 \in (0,1], 0<r<\min(\sqrt{t_0}, r_0)$ and $g(t)$ a solution to Ricci-DeTurck flow with background $g_0(t)$ on $$C(x_0,t_0,r):= \bigcup_{t\in (t_0-r^2,t_0)} B_{g_0(t)}(x_0,r)\times \{t\}$$ with $|g(t)-g_0(t)|_{g_0} \leq \varepsilon_0$. Then
\begin{equation*}
\big| (r\nabla^{g_0})^a (r^2\partial_t)^b(g-g_0)\big|_{g_0}(x_0,t_0) \leq C_{a,b}\ .
\end{equation*}
Furthermore, for every $k \in \mathbb{N}$ there exists $0<\varepsilon_k\leq\varepsilon_0$, such that if additionally $|g(t)-g_0(t)|_{g_0} \leq \varepsilon_k$ on $C(x_0,t_0,r)$, then there exists a constant $C_{a,b}'>0$
\begin{equation*}
\big| (r\nabla^{g_0})^a (r^2\partial_t)^b(g-g_0)\big|_{g_0}(x_0,t_0) \leq C'_{a,b} \sup_{C(x_0,t_0,r)}\big|g(t)-g_0(t)\big|_{g_0}\ ,
\end{equation*}
provided $a+2b \leq k$}.

We can assume that $r_0$ is sufficiently small, such that $g_0(t)$ is well controlled in a suitable coordinate system in $B_{g_0(0)}(p_0,r_0)$ for $0\leq t \leq 1$. The estimate then follows from local estimates for the Ricci-De Turck flow, see \cite[Proposition 2.5]{Bamler14}.\\[1ex]
{\bf Claim 2:} {\it There exists $0<\delta<1$, independent of $\gamma$ and $\Lambda$, such that for any $(x,t)\in D'$ we have}
$$C(x,t, (\delta t)^{1/2}) \subset D\ .$$
Note first the following basic estimate
\begin{equation*}
\begin{aligned}
 \frac{3}{2}\sqrt{\gamma t + (\Lambda +1)^2} +\sqrt{t/16} &\leq \frac{3}{2}\sqrt{\gamma t + (\Lambda +1)^2} + \sqrt{\gamma t/16}\\ &\leq \frac{3}{2}\sqrt{\gamma t + (\Lambda +1)^2} + \frac{1}{4}\sqrt{\gamma t +(\Lambda +1)^2}\\
 &= 2\sqrt{\frac{49}{64}\gamma t + \frac{49}{64}(\Lambda +1)^2} \leq 2\sqrt{\frac{49}{64}\gamma t + (\Lambda +1)^2}.
\end{aligned}
\end{equation*}
 Let $(x,t) \in D'$. By Lemma \ref{exp_con_est}, the function $\mathbf r$ satisfies
 $$\frac{1}{2} \leq |\nabla^{g_0(t)}\mathbf{r}|_{g_0(t)} \leq 2$$
 in $\{\mathbf{r}(x) \geq \sqrt{\gamma_0t +\Lambda_0^2}\}$. This, together with the previous estimate, implies there exists a $\delta >0$ such that 
\begin{equation*}
\begin{aligned}
B_{g_0(t')}(x,(\delta t)^{1/2} )&\subset \Big\{\mathbf{r} \leq \frac{3}{2}\sqrt{\gamma t + (\Lambda +1)^2} +\sqrt{t/16}\Big\}\\
&\subset \Big\{\mathbf{r} \leq 2\sqrt{\gamma t' + (\Lambda +1)^2}\Big\}\, 
  \end{aligned}
\end{equation*}
where $t'\in ((1-\delta)t,t) \subset ((49/64)t, t)$.\\[1ex]
{\bf Decay estimates in $D'$:} In the case that $0<t<1$, the estimates follow directly from claim 1 and 2.  Fix a point $(x_0,t_0) \in D'$. We can assume that $1<t_0\leq T$.  Let $\lambda:= 2/(t_0+1)$. Recall that we denote with $\varphi_t$ the diffeomorphisms which generate the Ricci flow $g_e(t) = t\varphi_t^*g_N$ of the expanding gradient soliton. We define
$$g^\lambda(t) = \lambda \varphi^*_\lambda g(\lambda^{-1}(t+1)-1)\ .$$
Note that this scaling is chosen such that $g_0^\lambda(t) = g_0(t)$. This implies that $g^{\lambda}$ solves Ricci-DeTurck flow with background $g_0(t)$ on 
$$D^\lambda=\left\{  (x,t)\in \varphi_{\lambda^{-1}}\Big\{\mathbf r(x)\leq 2 \sqrt{\gamma ((t+1)/\lambda -1)+(\Lambda+1)^2}\Big\}\subset N\times [0,1] \right\}$$
and the point $( x_0', 1)$, where $x_0':=\phi_{\lambda^{-1}}(x)$, corresponds to $(x_0,t_0)$ under this scaling. By claim 2 we see that 
$$D^\lambda \supset C(x_0',1,(\delta \lambda t_0)^{1/2}) \supset C(x_0',1, \delta^{1/2})\ .$$ We can thus apply claim 1 to obtain
\begin{equation*}
\big| (\nabla^{g_0})^a (\partial_t)^b(g^\lambda-g_0)\big|_{g_0}(x'_0,1) \leq \tilde{C}_{a,b} \end{equation*}
where $\tilde{C}_{a,b}=\delta^{-(a/2+b)}C_{a,b}$. Similarly
\begin{equation*}
\big| (\nabla^{g_0})^a (\partial_t)^b(g^\lambda-g_0)\big|_{g_0}(x'_0,1) \leq C''_{a,b} \sup_{C(x_0',1, \delta^{1/2}) }\big|g^\lambda(t)-g_0(t)\big|_{g_0} \ ,
\end{equation*}
where $C''_{a,b}=\delta^{-(a/2+b)}C'_{a,b}$.

 Since the norms are invariant under the diffeomorphism $\varphi_\lambda$, we obtain the desired estimates at $(x_0,t_0)$ by scaling back to $g(t)$.
\end{proof}

\begin{lemma}\label{annulus_estimate}
Let $(N,g_N,f)$ be an asymptotically conical gradient Ricci expander with positive curvature operator and let $(g_0(t))_{t\geq 0}$ be the induced Ricci flow with $g_0(0)=g_N$. There is a $C(g_N)>0$ such that for every $\Lambda\geq \Lambda_0$ the following holds. Define 
\begin{eqnarray*}
A&=&\left\{  (x,t)\in N\times [0,T],\; \sqrt{\gamma t+(\Lambda+1)^2} \leq \mathbf r(x)\leq 2\sqrt{\gamma t+(\Lambda+1)^2}  \right\}
\end{eqnarray*}
for some $\gamma\geq 1$, where $\mathbf r(x):=2\sqrt{f(x)}$. Then,  if the tensors  $h_1,h_2$ are supported in $A$ and satisfy $|h_1|_{g_0(t)} +|h_2|^2_{g_0(t)} \leq \frac{D}{t+\gamma^{-1}(\Lambda +1)^2}$ then 
\begin{eqnarray*}
 || h_1 +\nabla^{g_0(t)} h_2 ||_{Y_T} \leq C(g_N) D.
\end{eqnarray*} 
\end{lemma}
\begin{remark}
The importance of Lemma \ref{annulus_estimate} is that the constant $C(g_N)$ does not depend on $\Lambda$ or $\gamma$.
\end{remark}
\begin{proof}
We begin  by estimating the terms in the norms of $Y^0_{T'}$ and $Y^1_{T'}$ in two different cases. We will only present the computations for the norm of $h:=h_1$ in $Y^0_{T'}$ since the norm of $h_2$ in $Y^1_{T'}$ can be treated in a similar way. In the following $C(g_N)$ will denote a constant that depends only on the expander and is allowed to change from line to line.

To estimate the first term in $|| h||_{Y_{T'}^0}$, consider, first,  the following cases regarding $P(x,R)$:
\begin{itemize}
\item $P(x,R)\cap A \subset \{t\leq c_1 \gamma^{-1}R^2\}\cup  \{t\leq c_1 \gamma^{-1}(\Lambda+1)^2\}$
\begin{equation*}
\begin{aligned}
 &R^{-n}\int_{P(x,R)} |h(t)|_{g_0(t)} \de\mu_{g_0(t)} \de t \leq  \\ 
 & \leq D R^{-n} \int_0^{\frac{c_1 \max\{ R,\Lambda+1\}^2}{\gamma}} \int_{B_{g_0(t)}(x,R)\cap (A\cap N\times \{t\})} \frac{1}{t+\frac{(\Lambda +1)^2}{\gamma}} \de\mu_{g_0(t)} \de t.
\end{aligned} 
\end{equation*}
Now, for $R\geq \Lambda+1$ we estimate
\begin{equation}\label{estA1}
\begin{aligned}
R^{-n}&\int_{P(x,R)} |h(t)|_{g_0(t)} \de\mu_{g_0(t)} \de t \leq  \\ 
&\leq C(g_N) D R^{-n} \int_0^{\frac{c_1 R^2}{\gamma}} \frac{ (\gamma t+(\Lambda+1)^2)^{\frac{n}{2}}}{t+\gamma^{-1}(\Lambda +1)^2}  \de t, \\
&\leq C(g_N) D R^{-n} \gamma^{\frac{n}{2}} \left( \int_0^{\frac{c_1 R^2}{\gamma}} \left(t+\gamma^{-1}(\Lambda +1)^2\right)^{\frac{n}{2}-1}\de t \right),\\
&\leq C(g_N)  (c_1+1)^{\frac{n}{2}} D, 
\end{aligned}
\end{equation}
since  $\vol_{g_0(t)}(A\cap (N\times \{t\}))\leq C(g_N)  (\gamma t+(\Lambda+1)^2)^{\frac{n}{2}} $ from Lemma \ref{exp_con_est}.

  For $R<\Lambda+1$ we estimate as follows
\begin{equation}
\begin{aligned}
R^{-n}&\int_{P(x,R)} |h(t)|_{g_0(t)} \de\mu_{g_0(t)} \de t \leq  \\
&\leq C(g_N) D \int_0^{\frac{c_1 (\Lambda+1)^2}{\gamma}} \frac{\de t}{t+\gamma^{-1}(\Lambda +1)^2} \leq C(g_N)  \log(c_1+1) D, \label{estA2}
\end{aligned}
\end{equation}
where we also use that $\vol_{g_0(t)}(B_{g_0(t)}(x,R))\leq C(g_N) R^n$, which follows again from Lemma \ref{exp_con_est}.

\item $P(x,R)\cap A\subset \{\frac{\alpha R^2}{m}\leq t \leq \alpha R^2\}$, for some $\alpha \in (0,1]$. Then,
\begin{equation}\label{estB}
\begin{aligned}
 R^{-n}\int_{P(x,R)} |h(t)|_{g_0(t)} & \de\mu_{g_0(t)} \de t \leq\\
 &\leq R^{-n} \int_{\frac{\alpha R^2}{m}}^{\alpha R^2} \int_{B_{g_0(t)}(x,R)} |h(t)|_{g_0(t)} \de\mu_{g_0(t)} \de t \\
&\leq D  C(g_N) \int_{\frac{\alpha R^2}{m}}^{\alpha R^2}  \frac{\de t}{t} =D  C(g_N) \log m.
\end{aligned}
\end{equation}

\end{itemize}

For the second term in the definition of the norm of $Y^0_{T'}$ we can estimate directly:
\begin{equation}\label{estC}
\begin{aligned}
R^{\frac{4}{n+4}}& \left( \int_{Q(x,R)} |h|^{\frac{n+4}{2}} \de\mu_{g_0(t)} \de t  \right)^{\frac{2}{n+4}}\leq \\
& \leq D R^{\frac{4}{n+4}} \left( \int_{\frac{R^2}{2}}^{R^2}  \frac{C(g_N) R^n}{(t+\gamma^{-1}(\Lambda +1)^2)^{\frac{n+4}{2}}} \de t\right)^{\frac{2}{n+4}}\leq D C(g_N).
\end{aligned}
\end{equation}
Now, recall  the distance distortion estimate (\ref{dist_dist}) from Section \ref{prelim}:
\begin{eqnarray*}
 d_{g_0(0)}(x,y) - C(g_N) \sqrt t \leq d_{g_0(t)}(x,y).
\end{eqnarray*}
It implies that for $K=1+C(g_N)$, and every $x\in N$ and $t\in [0,R^2]$,
\begin{eqnarray*}
 B_{g_0(t)}(x,R)\subset B_{g_0(0)}(x,K R), 
\end{eqnarray*}
hence $P(x,R)\subset \hat P(x,R):= B_{g_0(0)}(x, KR) \times [0,R^2]$. 

Define 
\begin{align*}
S(x,R) &= \left\{ \mathbf r(x) - 2 K R \leq \mathbf r \leq \mathbf r(x) + 2 K R \right\}\times [0,R^2],\\
S(\Lambda,R)&= \left\{\mathbf r\leq 4(\Lambda +1) +4 K  R\right\}\times [0,R^2], 
\end{align*} 
and recall that 
\begin{equation}
\frac{1}{2}\leq|\nabla^{g_0(0)} \mathbf r|_{g_0(0)}<2 \label{grad_r}
\end{equation}
 in $\{\mathbf r\geq \Lambda\}$, by Lemma \ref{exp_con_est}, since $\Lambda\geq \Lambda_0$
 
We distinguish the following cases:
\begin{enumerate}
\item $\hat P(x,R)\subset S(\Lambda, R):$ In this case  let
\begin{equation*}
t_0=\max\{ t\in [0,R^2], A\cap S(\Lambda,R) \cap (N\times \{t\}) \neq \emptyset\}.
\end{equation*}
Then,
\begin{equation*}
\begin{aligned}
t_0&\leq\frac{( 4(\Lambda+1)+4 K R )^2}{\gamma}\\
&\leq ( 4 +4  K)^2 \frac{R^2}{\gamma}, \textrm{if } R\geq \Lambda+1\\
&\leq ( 4 +4 K )^2  \frac{(\Lambda+1)^2}{\gamma}, \textrm{if } R< \Lambda+1,
\end{aligned}
\end{equation*}
and the result follows from estimates \eqref{estA1}-\eqref{estA2} and \eqref{estC}.\\
\item  $\hat P(x,R)\not\subset S(\Lambda, R):$ In this case we can use \eqref{grad_r} to conclude that $\mathbf r(x) - 2 K R \geq 4(\Lambda+1)>\Lambda$ and $\hat P(x,R)\subset S(x,R)$.\\[1ex]
We may define
\begin{eqnarray*}
t_{in}&=&\min\{t\in [0,R^2],\; S(x,R)\cap A\cap (N\times \{t\})\neq \emptyset \}, \\
t_{out}&=&\max\{t\in [0,R^2],\; S(x,R)\cap A \cap (N\times \{t\})\neq \emptyset \},
\end{eqnarray*}
and note that $t_{in}>0$.\\[1ex]
Let $\alpha\in (0,1]$ such that  $t_{out}=\alpha R^2$. From  $$\sqrt{\gamma t_{out} +(\Lambda+1)^2} \geq \mathbf r(x)+2  K R$$ it follows that
\begin{eqnarray*}
 \mathbf r(x)\geq (\sqrt{\gamma \alpha}- 2 K)  R. 
\end{eqnarray*}
 Then, using $ \mathbf r(x)-2 K R=2\sqrt{\gamma t_{in}+(\Lambda+1)^2}$ we conclude that
\begin{equation}
\begin{aligned}
 t_{in}& \geq\frac{1}{4\gamma} (\sqrt{\gamma \alpha} - 4 K )^2 R^2 -\frac{(\Lambda+1)^2}{\gamma} \\
 &= \frac{\alpha R^2}{4} \left(\frac{(\sqrt{\gamma \alpha} -4 K  )^2 -\frac{4(\Lambda+1)^2}{R^2}}{\gamma\alpha}\right).
\end{aligned}
\end{equation}
Notice that if $\alpha>\gamma^{-1} \max\{(8(1+C(g_N)))^2, 32 R^{-2}(\Lambda+1)^2 \}$ it follows that $t_{in}\geq \frac{t_{out}}{32}$, and the result follows from estimates \eqref{estB}-\eqref{estC}. In any other case, either $t_{out} \leq C \gamma^{-1} R^2$ or $t_{out}\leq C\gamma^{-1}(\Lambda +1)^2$, therefore the result follows from estimates \eqref{estA1}-\eqref{estA2} and \eqref{estC}.
\end{enumerate}

\end{proof}

\begin{proof}[Proof of Lemma \ref{expanding_region}]
Suppose that $(M,g(0))\in \mathcal M(\alpha,\Lambda,s)$. Observe that the following identities hold regarding $Q=\Phi_s\circ \varphi_s^{-1}$:
\begin{eqnarray*}
Q^*r_s&=&\sqrt s \mathbf r,\\
s^{-1}Q^*\tilde g (st)&=&g_0(t) \textrm{ in } \Big\{\mathbf r \leq \frac{1}{2\sqrt s}\Big\}.
\end{eqnarray*}
Denoting $G(t)= s^{-1}Q^*\hat g(st)$ we obtain by the assumption on $\mathcal D^{cone}_{\gamma,\Lambda+1,s}$ that
\begin{equation*}
\begin{aligned}
| G(t) - g_0(t)|_{g_0(t)} +  \mathbf r |\nabla^{g_0(t)} G(t)&|_{g_0(t)}\leq Q^*(|\hat g-\tilde g|_{\tilde g}+  r_s |\tilde \nabla \hat g|_{\tilde g})(st)<\alpha,
\end{aligned}
\end{equation*}
in $\big\{\mathbf r\geq \sqrt{\gamma t+(\Lambda+1)^2}\big\}=Q^{-1}\big(\big\{ \sqrt{\gamma s t+s(\Lambda+1)^2} \leq r_s \leq 3/4 \big\}\big)$ for any \linebreak $t\in [0,s^{-1}\max\{   (32\gamma)^{-1}, T\}]$.

Moreover,
\begin{equation*}
 \big| G(0)-g_0(0) \big|_{g_0(0)}=  Q^*\big( | g(0)-\tilde g(0)|_{\tilde g(0)}\big) < \alpha
\end{equation*}
 in $\{\mathbf r\leq 2(\Lambda+2)\}$, since $(M,g(0))\in \mathcal M(\alpha,\Lambda,s)$.

Therefore,  by Lemma \ref{long_time_estimate}, for every $\varepsilon>0$ there is  $\alpha_0(g_N,\varepsilon,k)>0$ such that if $\alpha\leq\alpha_0$ then
\begin{eqnarray*}
\sup_{D'}\big|(t\partial_t)^a (t^{\frac{1}{2}}\nabla^{g_0(t)})^b (G(t)-g_0(t)) \big|_{g_0(t)}  < \varepsilon,
\end{eqnarray*}
for any non-negative indices $a,b$  with $a+2b\leq k$, where $D'$ is as in Lemma \ref{long_time_estimate}. Hence
\begin{eqnarray*}
\big|(t\partial_t)^a (t^{\frac{1}{2}}\tilde \nabla)^b (\hat g(t)-\tilde g(t)) \big|_{\tilde g(t)}(x) < \varepsilon,
\end{eqnarray*}
for $(x,t)\in M\times [0,\max\{ (32\gamma)^{-1}, T\}]$ satisfying $ r_s(x)\leq \frac{3}{2}\sqrt{\gamma t + s(\Lambda+1)^2}$, which suffices to prove the theorem.
\end{proof}

\section{Flowing metrics with conical singularities}\label{flowing_sing}
The aim of this section is to prove Theorem \ref{stexist} in the case of one conical singularity at $z_1\in Z$ modeled on the cone $(C(\mathbb S^{n-1}), g_c=dr^2 +r^2 g_1)$, with $\riem(g_1)\geq 1$ and $\riem(g_1)\not\equiv 1$, denoting the coordinate around $z_1$ of Definition \ref{conical_sing} by $\phi$. Since the arguments are local, the case of more than one singular points can be treated similarly. Then, we proceed to prove Theorem \ref{perturbation_thm}.

Let $(N,g_N,f)$ be the unique expander asymptotic to $(C(\mathbb S^{n-1}), g_c)$ given by \cite{Deruelle16}. Recall that it has strictly positive curvature operator, by Lemma \ref{exp_con_est}. Moreover, let $\kappa>0$ and $\Lambda_1\geq \Lambda_0$  be small and large constants respectively, which will be determined later in the course of the proof. By rescaling we may assume that $r_0=1$ and $k_Z(r)<\kappa$  for $r\in (0,1]$.

\subsection{The approximating sequence} \label{approx} Given any $s\in (0,\frac{1}{2}]$ let $ Z_s=Z\setminus \phi ((0, s^{1/4}) \times X)$ and $\mathbf r_s$ as in Section \ref{prelim}. Define the diffeomorphic manifolds 
\begin{equation*}
M_s = \frac{Z_s  \bigsqcup \{ \mathbf r_s \leq 1\} }{\{\phi(r,q) = F_{s}(r,q), r\in [s^{1/4},1]\}},
\end{equation*}
equipped with the  natural embeddings $\Phi_s : \{\mathbf r_s\leq 1 \} \rightarrow  M_s$, $\Psi_s: Z_s\rightarrow M_s$.
Also, define $r_s: M_s\rightarrow [0,1]$ as
\begin{equation*}
r_s(x)=\left\{
\begin{array}{ll}
\Lambda_1\sqrt{s}, & x\in\Phi_s(\{\mathbf r_s\leq \Lambda_1 \sqrt s\}),\\
(\Phi_s^{-1})^*\mathbf r_s, & x\in \Phi_s(\{  \Lambda_1 \sqrt s \leq \mathbf r_s \leq 1\}),\\
1, & x\in M_s\setminus \mathrm{Im}\,\Phi_s.
\end{array}
\right. 
\end{equation*}
and note that
\begin{equation}
r_s=((\Psi_s\circ\phi)^{-1})^*r \label{radial_fcn}
\end{equation}
in $\mathrm{Im}\, \Phi_s \cap \mathrm{Im}\,\Psi_s$.

Let $\xi_3$ be a smooth, positive and non-increasing function equal to $1$ in $(-\infty,1]$ and $0$ in $[2,+\infty)$. Now, we may define a Riemannian metric $G_s$ on  $M_s$ as follows 
\begin{equation*}
G_s = \xi_3(r_s/s^{1/4}) (\Phi_s^{-1})^*g_e(s) +(1-\xi_3(r_s/ s^{1/4})) (\Psi_s^{-1})^* g_Z.
\end{equation*}
In particular
 \begin{equation}\label{idata}
 G_s=\left\{
 \begin{array}{ll}
 (\Psi_s^{-1})^* g_Z   &\;\mathrm{in}\; \{r_s\geq 2s^{1/4}\}, \\
 (\Phi_s^{-1})^*g_e(s) &\;\mathrm{in}\; \{r_s\leq  s^{1/4}\}.
 \end{array}
 \right.
 \end{equation}

\subsection{Uniform almost conical behaviour.} \label{almost_cone}By the definition of $G_s$ it follows that there is $A$ such that
\begin{equation*}
\max_{\{r_s=1\}}  |  \riem(G_s)  |_{G_s} \leq A.
\end{equation*}

Let $\eta_0=\eta_0(g_N)$ be given by Theorem \ref{main_thm}. Then, choosing $\kappa$ small and $\Lambda_1$ large we obtain $(M_s,G_s)\in \mathcal M(\eta_0,\Lambda_1, s)$. For this, recall the computation (\ref{conv_to_cone}) and observe that, since $\Phi_s\circ F_s = \phi$ in $\{s^{1/4}\leq r_s <1\}$,
\begin{equation}\label{error}
\begin{aligned}
(\Phi_s&\circ F_{s})^* G_s-g_c=\\
&=\xi_3 (r_s/ s^{1/4})(F_{s}^* g_e(s)- g_c) + (1-\xi_3 (r_s/ s^{1/4})) (\phi^* g_Z - g_c),
\end{aligned}
\end{equation}
and that the support of $(\nabla^{g_c})^j \xi_3 (r_s/ s^{1/4})$, $j\geq 1$, is contained in $\{r_s\geq  s^{1/4}\}$.

\subsection{Taking the limit.}\label{limit} By Corollary \ref{corollary}, there exist $T,C_{\riem}>0$ such that for small $s$ the following hold for the Ricci flows $(h_s(t))_{t\in (0,T]}$ with $h_s(0)=G_s$:
\begin{align}
\max_{M_s} |\riem(h_s(t))|_{h_s(t)} &\leq \frac{C_{\riem}}{t},\, \mathrm{for}\; t\in (0,T], \label{5.curv_bound1}\\
\max_{M_s} \sum_{j=0}^2 r_s^{j+2} | (\nabla^{h_s(t)})^j \riem(h_s(t)) |_{h_s(t)} &\leq C_{\riem},\, \mathrm{for}\; t\in [0,T].\label{5.curv_bound2}
\end{align}
Moreover, 
\begin{equation*}
\vol_{h_s(t)} (B_{h_s(t)}(x,1))\geq v_0, \, \mathrm{for}\; t\in [0,T],
\end{equation*}
for some $x\in \{r_s=1\}$ due to (\ref{5.curv_bound2}).

Now, take any sequence $s_l\searrow 0$ and write $M_l=M_{s_l}$, $G_l=G_{s_l}$ and $h_l(t)=h_{s_l}(t)$. By Hamilton's compactness theorem applied to the sequence $(M_l,h_l(t))_{t\in [0,T]}$ we can obtain a compact and smooth Ricci flow $(M, g(t))_{t\in (0,T]}$ as a subsequential limit. Namely, there exist diffeomorphisms $H_l: M\rightarrow M_l$ such that 
\begin{equation}
H_l^* h_l(t) \rightarrow  g(t)\label{5.convergence}
\end{equation}
uniformly locally in $M\times (0,T]$ in the $C^\infty$ topology.

\subsection{The map $\Psi$.} Let $\tilde\Psi_l = H_l^{-1}\circ\Psi_l : Z_l :=Z_{s_l}\rightarrow M$. We will prove that there exists a map $\Psi:Z\setminus \{z_1\} \rightarrow M$, diffeomorphism onto its image, such that $\tilde\Psi_l$ converges to $\Psi$ in $C^\infty$ uniformly away from $z_1$. Since $M$ is compact and $Z_l\subset Z_{l+1}$ exhaust $Z\setminus \{z_1\}$, it suffices to obtain derivative estimates for $\tilde\Psi_l$ and $\tilde \Psi_l^{-1}$ with respect to fixed metrics on $Z\setminus \{z_1\}$ and $M$.

First, observe that around any  $ p\in Z_l$ and $\Psi_l(p)\in \mathrm{Im}\,(\Psi_l)$ there are local coordinates $\{x^k\}_{k=1,\ldots,n}$ and $\{y^k\}_{k=1,\ldots,n}$ respectively, so that
\begin{equation}\label{identity}
x^k=\Psi_l^* y^k,
\end{equation}
and
\begin{equation}\label{5.metric_c}
\begin{aligned}
2^{-1}\delta \leq g_Z &\leq 2\delta,&\quad 2^{-1}\delta \leq h_l(0) &\leq 2\delta,\\
\left|\frac{\partial^j (g_Z)_{pq}}{\partial x^{k_1}\cdots \partial x^{k_j}} \right| &\leq C_{j,l},&\quad \left|\frac{\partial^j h_l(0)_{pq}}{\partial y^{k_1}\cdots \partial y^{k_j}} \right| &\leq C_{j,l},
\end{aligned}
\end{equation}
for all $j$, since $\Psi_l^* h_l(0)=g_Z$ in $Z_l$, by \eqref{idata}. Here $\delta$ denotes the Euclidean metric in the corresponding coordinates.

Applying \eqref{idata}, Perelman's pseudolocality theorem and Shi's local derivative estimates to $(h_l(t))_{t\in [0,T)}$, as in the proof of Lemma \ref{conical_region}, together with the bound \eqref{5.curv_bound1}, we obtain the following: for every $l_0$ and any non-negative index $j$ there exist $C_{j,l_0}$ such that for $l\geq l_0$
\begin{eqnarray}\label{improved_curv_bound}
 |(\nabla^{h_l(t)})^j \riem(h_l(t))|_{h_l(t)} \leq C_{j,l_0},
\end{eqnarray}
in $\mathrm{Im}\,(\Psi_l|_{Z_{l_0}})\subset \{ r_l \geq 2 s_{l_0}^{1/4} \}$ and $t\in [0,T]$. Thus, in $\mathrm{Im}\,(\Psi_l|_{Z_{l_0}})$,
\begin{equation}\label{5.hl_control}
\begin{aligned}
Q_{l_0}^{-1}h_l(0)\leq h_l(T) \leq Q_{l_0} h_l&(0)\\
\left| \frac{\partial^j h_l(T)_{pq}}{ \partial y^{k_1}\cdots \partial y^{k_j}}  \right| \leq Q_{j,l_0}&, 
\end{aligned}
\end{equation}
for any $l\geq l_0$ and non-negative $j$.

Then  \eqref{identity}, \eqref{5.metric_c} and \eqref{5.hl_control} imply that
 \begin{equation*}
 \begin{aligned}
 \left| (\nabla^{g_Z, h_l(T)})^j \Psi_l |_{Z_{l_0}} \right|_{{g_Z, h_l(T)}}&\leq C'_{j,l_0} \\
 \left| (\nabla^{h_l(T), g_Z})^j \Psi_l^{-1} |_{\Psi_l(Z_{l_0})} \right|_{ h_l(T) ,g_Z}&\leq C'_{j,l_0},
 \end{aligned}
 \end{equation*}
 for any non-negative $j$.
 
   Finally, since $H_l^*h_l(T)\rightarrow g(T)$, we obtain 
  \begin{equation*}
  \begin{aligned}
  \left|(\nabla^{g_Z, g(T)})^j \tilde\Psi_l |_{Z_{l_0}} \right|_{g_Z, g(T)}&\leq C''_{j,l_0}\\
  \left|(\nabla^{g(T),g_Z})^j \tilde\Psi_l^{-1} |_{\tilde \Psi_l(Z_{l_0})} \right|_{g(T),g_Z}&\leq C''_{j,l_0},
  \end{aligned}
  \end{equation*}
 for any non-negative $j$. The existence of $\Psi$ follows from Arzel\`{a}--Ascoli.

\subsection{Curvature bounds for the limit.} Since $(M_l,h_l(t))_{t\in(0,T]}$ satisfy \eqref{5.curv_bound1}, it is clear that  $g(t)$ satisfies
\begin{equation}
|\riem(g(t))|_{g(t)}\leq \frac{C_{\riem}}{t}\label{typeI}
\end{equation}
on $M\times (0,T]$. 

Now, notice that $H_l^* r_l = (\Psi_l^{-1}\circ H_l)^* (\phi^{-1})^*r$ in $(H_l^{-1}\circ \Psi_l)(Z_l)$, by \eqref{radial_fcn}. By  $\tilde\Psi_l^{-1}\rightarrow \Psi^{-1}$ it follows that
\begin{equation}
H_l^* r_l \rightarrow (\Psi^{-1})^* [(\phi^{-1})^* r], \label{radial_convergence}
\end{equation}
in $C^\infty_{loc, g(T)}(\mathrm{Im}\, \Psi)$. Recall that $\phi$ parametrises the conical region in $Z$.
 
Let $r_M$ be the continuous function on $M$ defined as
\begin{equation*} 
r_M=\left\{\begin{array}{cc}
[(\Psi \circ  \phi)^{-1}]^* r& \mathrm{in}\; \mathrm{Im}\,(\Psi\circ \phi),\\
0 & \mathrm{in}\; (\mathrm{Im \Psi})^c,\\
1 & \mathrm{otherwise}.
\end{array}\right.
\end{equation*}
By \eqref{5.curv_bound2} and \eqref{radial_convergence} it follows that $g(t)$ satisfies
\begin{equation}
\max_M \sum_{j=0}^2 r_M^{j+2}| (\nabla^{g(t)})^j \riem(g(t))|_{g(t)}\leq C_{\riem}, \label{5.curv_bound3}
\end{equation}
in $M\times (0,T]$.

\subsection{Uniform convergence to the initial data, away from the singular point.} \label{smoothly_data} Observe that 
\begin{eqnarray}
\Psi_l^* h_l(t) &=&(\tilde\Psi_l)^*(H_l^* h_l(t)) \label{pullback}\\
\Psi_l^* h_l(0)&=&g_Z. \label{initial_condition}
\end{eqnarray}
Since $\tilde\Psi_l\rightarrow \Psi$ and $H_l^* h_l(t)\rightarrow g(t)$, \eqref{pullback} implies that $\Psi_l^* h_l(t) \rightarrow \Psi^*g(t)$. 

Finally, the curvature bound \eqref{improved_curv_bound} and relation \eqref{initial_condition} imply that $\Psi^*g(t)$ converges  to $g_Z$ as $t\rightarrow 0$, in $C^\infty_{loc}$.
 
\subsection{Closeness to expander improves in small scales.} We will need the following lemma regarding the flows $(M_s,h_s(t))_{t\in (0,T]}$. 

\begin{lemma}\label{close_to_expander}
 For every $\varepsilon>0$ and integer $k\geq 0$, there exist positive $\lambda_1(\varepsilon,k),s_2(\varepsilon,k)$ small and $\gamma_3(\varepsilon,k),\Lambda_2(\varepsilon,k)$ large such that the following holds. For each $s\in (0,s_2]$, $\gamma\geq\gamma_3$ and $t\in(0,\lambda_1(32\gamma)^{-1}]$ there is a map 
 $$Q_{s,t}: \Big\{ r_s \leq \frac{5}{4}\sqrt{\gamma t+s(\Lambda_2+1)^2} \Big\} \rightarrow N,$$
 diffeomorphism onto its image, such that for all non-negative integers $j\leq k$,
\begin{equation*}
(t+s)^{j/2}\big|(\nabla^{g_e(t+s)})^j [(Q_{s,t}^{-1})^* h_s(t) - g_e(t+s) ]\big|_{g_e(t+s)} <\varepsilon,
\end{equation*}
in $\mathrm{Im}\,Q_{s,t}$ and $ \big\{ \mathbf r_s\leq \sqrt{\gamma t}  \big\} \subset \mathrm{Im}\, Q_{s,t} \subset \big\{ \mathbf r_s\leq \frac{3}{2}\sqrt{\gamma t +s(\Lambda_2+1)^2}  \big\}$.
\end{lemma}

\begin{remark}\label{lambda_Gamma} In the above statement we can assume w.l.o.g that $\gamma_3(\varepsilon,k)\geq(\Lambda_2(\varepsilon,k)+1)^2$.
\end{remark}

 \begin{proof} Given any $\varepsilon>0$, let $\eta_1=\eta_1(g_N,\varepsilon,k)$ be the constant provided by Theorem \ref{main_thm}. 

Since $\lim_{r\rightarrow 0} k_Z(r)=\lim_{r\rightarrow +\infty} k_{exp}(r)=0$, there are $\lambda_1(\varepsilon)>0$ small and $\Lambda_2(\varepsilon)>0$ large such that
\begin{equation*}
\begin{aligned}
k_Z(r) &<\eta_1,\; \mathrm{for}\; r\in (0,\lambda_1^{1/2}],\\
k_{exp}(r) &< \eta_1,\; \mathrm{for}\; r\geq \Lambda_2.
\end{aligned}
\end{equation*}
Moreover, set $s_2(\varepsilon,k)=\min\left\{ 2^{-4} \lambda_1^2,\big(2(\Lambda_2+1)\big)^{-4}\right\}$, we have   $$2(\Lambda_2 +1)\sqrt{s}\leq s^{1/4} < 2 s^{1/4}\leq \lambda_1^{1/2},$$
for every $s\in(0,s_2]$.

By construction of $(M_s,G_s)$ it follows that $(M_s,\lambda_1^{-1}G_s)\in \mathcal M(\eta_1,\Lambda_2,s/\lambda_1)$ for any  $s\in (0,s_2]$, with associated map $\Phi_{s/\lambda_1}= \Phi_s\circ \varphi_{\lambda_1^{-1}}$ and  function $r_{s/\lambda_1}=\max\{\Lambda_2\sqrt{s/\lambda_1},\min\{\lambda_1^{-1/2} r_s,1\} \}$. 

Theorem \ref{main_thm} implies that there is $\gamma_1>1$ such that for every $\gamma\geq \gamma_1$ and $\tau\in(0, (32\gamma)^{-1}]$ the metric $\lambda_1^{-1}h_s(\lambda_1\tau)$ is $\varepsilon$-close to  $g_e(\tau+s/\lambda_1)$ in 
\begin{equation*}
\Big\{ r_{s/\lambda_1}\leq \sqrt{\gamma \tau+ s \Lambda_2^2/\lambda_1}\Big\}= \Big\{ r_s \leq \sqrt{\gamma \lambda_1 \tau + s \Lambda_2^2}\Big\}.
\end{equation*}
 
 Then, for every $t\in (0,\lambda_1 (32\gamma)^{-1}]$  apply the above for $\tau=t/\lambda_1$ to prove the lemma for $\gamma_3=\gamma_1$.
 \end{proof}
 
\subsection{Diameter control of high curvature region of $g(t)$.} We will prove the following lemma.
\begin{lemma}[High curvature-small diameter]\label{diam_est}
There exists $c_0>0$ with the following property: for small $\zeta>0$ there exists $C_\zeta>0$ such that if $t\in(0,c_0 \zeta]$ then
\begin{eqnarray*}
\diam_{g(t)}\big(\{r_M\leq \sqrt{\gamma t} \}\big) &\leq& C_\zeta \sqrt t, \\
 |\riem(g(t))|_{g(t)}&<&\frac{\zeta}{t}\ \ \ \mathrm{in}\; \{r_M>\sqrt{\gamma t}\},
\end{eqnarray*}
where $C_\zeta=C(g_N)C_{\riem}^{1/2}\zeta^{-1/2}$ and $\gamma=C_{\riem}\zeta^{-1}$.
\end{lemma}
\begin{proof}
Fix $\varepsilon=10^{-2}$. By \eqref{5.curv_bound3} and putting $k=0$ in Lemma \ref{close_to_expander} we can find $\Lambda_2,\lambda_1$  such that, if $\gamma=C_{\riem}\zeta^{-1}$ and $\zeta$ is small then:
\begin{itemize}
\item For large $l$ and  each $t\in (0,\lambda_1(32\gamma)^{-1}]$ there exists 
$$Q_{l,t}: \Big\{r_l\leq \frac{5}{4}\sqrt{\gamma t+s_l(\Lambda_2+1)^2}\Big\}\rightarrow N$$ satisfying 
\begin{equation*}
\big|(Q_{l,t}^{-1})^* h_l(t) -g_e(t+s_l)\big|_{g_e(t+s_l)}< 10^{-2} 
\end{equation*}
in $\mathrm{Im}\, Q_{l,t}\subset \big\{\mathbf r_l\leq \frac{3}{2}\sqrt{\gamma t+s_l(\Lambda_2+1)^2}\big\}$.\\
\item $|\riem(g(t))|_{g(t)}\leq \frac{C_{\riem}}{r_M^2} < \frac{C_{\riem}}{\gamma t}=\frac{\zeta}{t}$ in $ \{r_M>\sqrt{\gamma t}\}$ provided that $t\in (0,\lambda_1(32\gamma)^{-1}]$.
\end{itemize}

By the closeness to the expander we obtain:
  \begin{equation*}
\begin{aligned}
\diam_{h_l(t)}&\big( \big\{r_l\leq \sqrt{\gamma t+s_l(\Lambda_2+1)^2}\big\}\big) \leq \diam_{(Q_{l,t}^{-1})^* h_l(t)}(\mathrm{Im}\,Q_{l,t})\\
&\leq (1.01)^{1/2} \diam_{g_e(t+s_l)}(\mathrm{Im}\, Q_{l,t})\\
&\leq (1.01)^{1/2} \diam_{g_e(t+s_l)}\Big( \Big\{\mathbf r_l\leq \frac{3}{2}\sqrt{\gamma t+s_l(\Lambda_2+1)^2}\Big\}\Big).
\end{aligned}
\end{equation*}

Working on the expander we compute, using Lemma \ref{exp} below for the last inequality,
\begin{equation}\label{work_on_expander}
\begin{aligned}
\diam&_{g_e(t+s_l)}\Big( \Big\{\mathbf r_l \leq \frac{3}{2}\sqrt{\gamma t+s_l(\Lambda_2+1)^2} \Big\}\Big)\\
&=\sqrt{t+s_l}\diam_{\varphi_{t+s_l}^*g_N} \Big(   \Big\{\mathbf r_l \leq \frac{3}{2}\sqrt{\gamma t+s_l(\Lambda_2+1)^2} \Big\} \Big )\\
&=\sqrt{t+s_l}\diam_{\varphi_{t+s_l}^*g_N} \Big( \varphi_{s_l}^{-1} \Big( \Big\{\mathbf r \leq \frac{3}{2}\sqrt{\gamma \frac{t}{s_l}+(\Lambda_2+1)^2} \Big\}\Big )\Big)\\
&= \sqrt{t+s_l} \diam_{g_N} \Big( \varphi_{1+\frac{t}{s_l}}\Big( \Big\{ \mathbf r\leq \frac{3}{2} \sqrt{\gamma \frac{t}{s_l} +(\Lambda_2+1)^2 } \Big\} \Big )  \Big  )\\
&\leq  C_\zeta\sqrt{t+s_l},
\end{aligned}
\end{equation}
where $C_\zeta= C(g_N)C_{\riem}^{1/2} \zeta^{-1/2}$. 

Now note that
\begin{equation}
\begin{aligned}
\diam&_{h_l(t)}\big(\big\{ r_l \leq \sqrt{\gamma t+s_l(\Lambda_2+1)^2}    \big\}\big)\\
&=\diam_{H_l^*h_l(t)}\big(\big\{ H_l^* r_l \leq \sqrt{\gamma t+s_l(\Lambda_2+1)^2}    \big\}\big)\\
&=\diam_{H_l^*h_l(t)}\big(\big\{ H_l^* (\Psi_l^{-1})^* (\phi^{-1})^* r \leq \sqrt{\gamma t+s_l(\Lambda_2+1)^2}    \big\}\big)\\
&=\diam_{H_l^*h_l(t)}\big(\big\{  (\Psi_l^{-1}\circ H_l)^* (\phi^{-1})^* r \leq \sqrt{\gamma t+s_l(\Lambda_2+1)^2}    \big\}\big),
\end{aligned}
\end{equation}
where we also used \eqref{radial_fcn}.

Since $H_l^*h_l(t)\rightarrow g(t)$ and $\Psi_l^{-1}\circ H_l \rightarrow \Psi^{-1}$, it follows that
\begin{equation*}
\diam_{g(t)}(\{ r_M\leq \sqrt{\gamma t} \} ) \leq C_\zeta \sqrt{t}.
\end{equation*}

\end{proof}

\begin{lemma}\label{exp}
Let $(N,g_N, f)$ be a gradient Ricci expander with bounded curvature. Denote by $R_{inf}, R_{sup}$ the infimum and supremum of the scalar curvature respectively and suppose $f$ is normalised so that $|\nabla f|^2=f+R_{inf}-R$. Let $\mathbf r=2\sqrt{f}$ and $\varphi_{1+u}$ be the associated family of diffeomorphisms. Then, if $\gamma \geq  (\Lambda+1)^2 \geq 32 (R_{sup}-R_{inf})$, then  
\begin{equation}
\varphi_{1+u}\Big(  \Big\{  \mathbf r\leq \frac{3}{2}\sqrt{\gamma u +  (\Lambda+1)^2}  \Big\} \Big) \subset \big\{ \mathbf r\leq \sqrt{8\gamma}\big\},\label{exp_one}
\end{equation}
for all $u\geq 0$ and
\begin{equation}
\varphi_{1+u}\Big( \Big\{ \mathbf r\leq \frac{1}{2}\sqrt{\gamma u+(\Lambda+1)^2} \Big\}   \Big)\supset \big\{ \mathbf r \leq \sqrt{ \gamma/8}\big\}\label{exp_two},
\end{equation}
for $u\geq 1$.
\end{lemma}
\begin{proof}
First note that the normalisation of $f$ implies that 
$$f=|\nabla f|^2 +R-R_{inf}\geq 0,$$
and $f>0$ away from the critical points of $f$.

By \eqref{diffeo_evol} it follows that 
\begin{equation}\label{ODE1}
\frac{d}{du} f\circ \varphi_{1+u} =-\frac{1}{1+u}|\nabla f|^2\circ \varphi_{1+u}.
\end{equation}
In order to prove \eqref{exp_two}  note that, since  $|\nabla f|^2 = f+R_{inf}-R\leq f$, \eqref{ODE1} becomes
\begin{equation*}
\frac{d}{du} f\circ \varphi_{1+u} \geq -\frac{1}{1+u} f\circ \varphi_{1+u}.
\end{equation*}
Integrating this inequality we immediately obtain that 
\begin{equation}\label{lbound}
f\circ \varphi_{1+u}(x) \geq \frac{f(x)}{1+u},
\end{equation}
for all $x\in N$ with $\nabla f (x)\not=0$ and $u\geq 0$. 

Thus, if  $x$ is such that $\mathbf r(x) \geq \frac{1}{2}\sqrt{\gamma u +(\Lambda+1)^2}$ it follows that
\begin{equation}
\mathbf r(\varphi_{1+u}(x))\geq \frac{1}{2\sqrt{2}}(\Lambda+1), \label{bu1}
\end{equation}
for $0\leq u\leq1$ and
\begin{equation}
\mathbf r(\varphi_{1+u}(x))\geq \sqrt{\gamma/8}, \label{bu2}
\end{equation}
 for $u\geq 1$, which proves \eqref{exp_two}.

On the other hand $|\nabla f|^2\geq f-C(g_N)$, where $C(g_N)=R_{sup} - R_{inf}>0$, hence \eqref{ODE1} becomes
\begin{equation*}
\frac{d}{du} f\circ \varphi_{1+u} \leq -\frac{1}{1+u} (f-C(g_N))\circ \varphi_{1+u}.
\end{equation*}
Hence,  as long as $f\circ \varphi_{1+u}(x) \geq C(g_N)$,  $f\circ \varphi_{1+u}(x)$ is non-increasing  in  $u$ and 
\begin{equation}\label{ubound}
f\circ\varphi_{1+u}(x)  \leq \frac{1}{1+u} (f(x) - C(g_N)) +C(g_N).
\end{equation}
Thus, if $x$ is such that $\mathbf r(x)=\frac{3}{2} \sqrt{\gamma u +(\Lambda+1)^2}$ and $\gamma \geq (\Lambda+1)^2\geq 32 C(g_N)$, by  \eqref{bu1} and \eqref{bu2}:
$$ f\circ \varphi_{1+u}(x) \geq 
\left\{ 
\begin{array}{ll}
\gamma/32, &\textrm{if $u\geq 1$} \\
\left(\Lambda+1\right)^2/32, &\textrm{if $0\leq u \leq 1$}
\end{array}
\right.
\geq C(g_N)$$
for $u\geq 0$. Hence, by \eqref{ubound} and $\gamma \geq (\Lambda+1)^2\geq 32 C(g_N)$,
\begin{eqnarray*}
f\circ \varphi_{1+u}(x) &\leq& \frac{1}{1+u} f(x) +C(g_N) \\
&\leq& \frac{9}{16} (\gamma +(\Lambda+1)^2 ) +C(g_N)\\
&\leq& 2\gamma.
\end{eqnarray*}

It follows that
$\mathbf r\circ \varphi_{1+u}(x) \leq \sqrt{8\gamma},$
which proves \eqref{exp_one}.
\end{proof}

\subsection{Gromov--Hausdorff convergence to the initial data.}\label{GH} In this section we prove that for every $\varepsilon>0$ the map $\Psi: Z\setminus \{z_1\} \rightarrow M$ is an $\varepsilon$-isometry  between $(Z\setminus\{ z_1\},d_Z)$ and $(M,g(t))$ for small $t$, which implies that $(M,d_{g(t)})$ converges to $(Z,d_Z)$ in the Gromov--Hausdorff sense as $t\rightarrow 0$. 

The result follows immediately from the following two lemmata:

\begin{lemma}[Distortion estimate] \label{distortion}
For every $\varepsilon>0$  there exists $\delta_1, t_1>0$, such that the map
\begin{equation}
\Psi: \{r\geq \delta_1 \} \rightarrow \{r_M\geq \delta_1\},
\end{equation}
satisfies 
\begin{equation}
\sup \left\{ |d_{g(t)}(\Psi(z_1),\Psi(z_2)) - d_Z(z_1,z_2)|,\; z_1,z_2\in \{r\geq \delta_1 \}  \right\}<3\varepsilon,
\end{equation}
for every $t\in (0,t_1]$, and $\diam(\{r\leq\delta_1\})<\varepsilon$.
\end{lemma}
\begin{proof}
Take $\delta_1>0$ such that the intrinsic (hence also the extrinsic) diameter
\begin{equation}
\diam_{g_Z}(\{r=\delta_1\})<\varepsilon.\label{diam_d1}
\end{equation}
 By the uniform convergence away from $z_1$, as $t\rightarrow0$, it follows that 
 \begin{equation}
 \diam_{g(t)}(\{r_M=\delta_1\})<\varepsilon\label{diam_d2}
 \end{equation}
 for small $t$.

We will use $d_{g_Z,\delta_1}$ to denote the intrinsic metric in $\{r\geq\delta_1\}$ induced by $g_Z$, and similarly $d_{g(t),\delta_1}$ for the intrinsic metric in $\{r_M\geq\delta_1\}$ induced by $g(t)$.

By \eqref{diam_d1} and  \eqref{diam_d2}, it follows that for every $z_1,z_2\in \{r\geq \delta_1\}$
\begin{eqnarray}
| d_{g_Z,\delta_1}(z_1,z_2) - d_Z(z_1,z_2)| &<&\varepsilon,\label{dist1}\\
| d_{g(t),\delta_1}(\Psi(z_1),\Psi(z_2)) - d_{g(t)}(\Psi(z_1),\Psi(z_2))| &<&\varepsilon.\label{dist2}
\end{eqnarray}
To see this, note for instance that
$$ d_{g_Z}(z_1,z_2)\leq  d_{g_Z,\delta_1}(z_1,z_2) \leq d_{g_Z}(z_1,\{ r=\delta_1\})+d_{g_Z}(z_2,\{ r=\delta_1\})+\varepsilon.$$

Moreover, if
$$ d_{g_Z}(z_1,\{ r=\delta_1\})+d_{g_Z}(z_2,\{ r=\delta_1\}) > d_{g_Z}(z_1,z_2)$$
then $ d_{g_Z}(z_1,z_2)= d_{g_Z,\delta_1}(z_1,z_2)$. For, if $ d_{g_Z}(z_1,z_2)< d_{g_Z,\delta_1}(z_1,z_2)$, then there is a path connecting $z_1,z_2$ escaping $\{r\geq\delta_1\}$, hence
 $$ d_{g_Z}(z_1,z_2) > d_{g_Z}(z_1,\{ r=\delta_1\})+d_{g_Z}(z_2,\{ r=\delta_1\}),$$
 which is a contradiction. This proves \eqref{dist1}, and \eqref{dist2} is similar.
 
By the uniform convergence away from $z_1$, as $t\rightarrow0$, it also follows that for small $t$
\begin{equation}
|d_{g_Z,\delta_1}(z_1,z_2) - d_{g(t),\delta_1}(\Psi(z_1),\Psi(z_2)) |<\varepsilon,\label{dist3}
\end{equation}
uniformly for all $z_1,z_2\in \{r\geq \delta_1\}$. The result follows from the triangle inequality, combining \eqref{dist1}-\eqref{dist3}, having possibly made $\delta_1>0$ smaller in order to achieve $\diam(\{r\leq\delta_1\})<\varepsilon$.

\end{proof}
 
 \begin{lemma}[$\mathrm{Im \Psi}$ is an $\varepsilon$-net]\label{net}
 For every $\varepsilon>0$ and small enough $\delta_2,t_2>0$
 \begin{equation}
 \diam_{g(t)}(\{r_M\leq \delta_2\})<\varepsilon, \label{diam_c3}
 \end{equation}
 for every $t\in (0,t_2]$. 
 \end{lemma}
 \begin{proof}
  Let $c_0$ be the constant given by Lemma \ref{diam_est}. Then, since  $\frac{c_0 C_{\riem}}{r_0^2}>1$ for small $r_0$, it follows that $t\in (0, \frac{c_0 C_{\riem}}{r_0^2} t]$ hence we can apply Lemma \ref{diam_est} for $\zeta=\frac{C_{\riem}t}{r_0^2}$ to obtain
 \begin{equation*}
 \diam_{g(t)}(\{r_M\leq r_0\}) \leq C(g_N) r_0, 
 \end{equation*}
 for small $t$, which proves the lemma.
\end{proof}

\subsection{Tangent flow at the conical point.} \label{tangent flow}Take any sequence of times $t_k\searrow 0$. It follows from the convergence \eqref{5.convergence} that there is a sequence $l_k$ such that for any non-negative index $j\leq k$
\begin{equation*}
t_k^{j/2}|(\nabla^{g})^{j}(  g- H_{l_k}^*  h_{l_k})|_{g}(t_k) < 1/k\;\; \mathrm{and}\;\;
s_{l_k}/t_k \rightarrow 0.
\end{equation*}
Let $\gamma_k=\gamma_3(1/k,k)$, $\Lambda_k=\Lambda_2(1/k,k)$ and $\lambda_k=\lambda_1(1/k,k)$ as given by Lemma \ref{close_to_expander} and set $\tau_k=\lambda_k(32 \gamma_k)^{-1}$. Passing to a subsequence if necessary, we may assume that $t_k < \tau_k$  and $s_{l_k}<s_2(1/k,k)$.

By Lemma \ref{close_to_expander}, there exist
\begin{equation*}
Q_k : \Big\{r_{l_k}\leq \sqrt{\gamma_k t_k +s_{l_k}(\Lambda_k+1)^2} \Big\}\rightarrow N,
\end{equation*} 
diffeomorphisms onto their image, such that for $j\leq k$
\begin{equation*}
(t_k +s_{l_k})^{j/2} \big|  (\nabla^{g_e(t_k+s_{l_k})})^j  ((Q_k^{-1})^* h_{l_k}(t_k) - g_e(t_k+s_{l_k})  )  \big|_{g_e(t_k+s_{l_k})} < 1/k
\end{equation*}
in $\mathrm{Im}\, Q_k$. Thus, setting $R_k=(Q_k\circ H_{l_k})^{-1}$,  we obtain
\begin{equation*}
t_k^{j/2} \big| (\nabla^{g_e(t_k + s_{l_k})})^j ( R_k^* g(t_k)- g_e(t_k+s_{l_k}) )  \big|_{g_e(t_k+s_{l_k})}  <C/k
\end{equation*}
in $\mathrm{Im}\, Q_k$, for large $k$. Moreover, since
\begin{equation*}
t_k^{-1} g_e(t_k+s_{l_k})=(1+\frac{s_{l_k}}{t_k})\varphi_{t_k+s_{l_k}}^* g_N,
\end{equation*}
we conclude that
\begin{equation*}
\big|  (\nabla^{g_N})^j  (  (R_k  \circ\varphi_{t_k+s_{l_k}}^{-1})^*  t_k^{-1} g(t_k) - (1+s_{l_k}/t_k) g_N )  \big|_{g_N} < C/k
\end{equation*}
in $\varphi_{t_k+s_{l_k}}(\mathrm{Im}\, Q_k)$.

Putting $G_k=  (R_k\circ\varphi_{t_k+s_{l_k}}^{-1} )^* t_k^{-1}g(t_k) $, the estimate above becomes
\begin{equation}
\big|(\nabla^{g_N})^j G_k-(1+s_{l_k}/t_k)g_N\big|_{g_N} <C / k \label{conv_to_exp}
\end{equation}
in $\mathrm{Im}(\varphi_{t_k+s_{l_k}}\circ R_k^{-1})=\varphi_{t_k+s_{l_k}}(\mathrm{Im}\, Q_k)$.

Then, since by Lemma \ref{close_to_expander} and Remark \ref{lambda_Gamma}
\begin{equation*}
\Big \{\mathbf r_{l_k}\leq \frac{1}{2}\sqrt{\gamma_k t_k +s_{l_k} (\Lambda_k+1)^2}\Big\}\subset \Big \{\mathbf r_{l_k}\leq \sqrt{\gamma_k t_k }\Big\} \subset \mathrm{Im}\, Q_k,
\end{equation*}
it follows that
\begin{equation}\label{domain}
\begin{aligned}
\varphi_{t_k+s_{l_k}}(\mathrm{Im}\,Q_k)&\supset \varphi_{t_k+s_{l_k}} \Big(\Big\{ \mathbf r_{l_k} 
\leq \frac{1}{2}\sqrt{\gamma_k t_k +s_{l_k}(\Lambda_k+1)^2} \Big\} \Big)\\
&=\varphi_{1+ t_k/s_{l_k}} \Big(\Big\{ \mathbf r\leq \frac{1}{2}\sqrt{\gamma_k t_k/s_{l_k}+ (\Lambda_k+1)^2}\Big\}\Big)\\
&\supset \big\{\mathbf r \leq \sqrt{\gamma_k/8 } \big\},
\end{aligned}
\end{equation}
where the last inclusion follows from Lemma \ref{exp}.

Now, let $q_k\in M$ be such that $q_{max}=\varphi_{t_k+s_{l_k}}\circ R_k^{-1}(q_k)\in N$ satisfies 
$$|\riem(g_N)(q_{max})|_{g_N}=\max_N |\riem(g_N)|_{g_N}.$$
Applying Lemma \ref{diam_est} for $\zeta=\frac{1}{2}\max_N |\riem(g_N)|_{g_N}$ we obtain $\hat C,\hat\gamma>1$ such that
$$q_k\in \big\{r_M\leq \sqrt{\hat\gamma t_k}\big\},$$
and $\diam_{g(t_k)}(\{r_M\leq \sqrt{\hat\gamma t_k}\}) \leq \hat C \sqrt{t_k}$. 

 Given any $p_k\not\in \mathrm{Im}\,\Psi$, it follows that $r_M(p_k)=0$, hence $\dist_{g(t_k)}(p_k,q_k)\leq \hat C\sqrt{t_k}$. Therefore, 
 $\dist_{g_N} (q_{max}, \varphi_{t_k+s_{l_k}}\circ R_k^{-1}(p_k)) \leq 2 \hat C$, for large $k$.
 
 This, together with \eqref{conv_to_exp}, \eqref{domain} and that $\gamma_k \rightarrow +\infty$ suffices to prove that $(M,t_k^{-1}g(t_k),p_k)$ converges in the smooth pointed Cheeger--Gromov topology to $(N,g_N, \bar q)$. 
 
 This implies that $(M,t_k^{-1}g(t_k t), p_k)_{t\in (0,t_k^{-1}T]} \rightarrow (N,h(t),\bar q)_{t\in(0,+\infty)}$  in the smooth pointed Cheeger--Gromov topology, where $(N,h(t))$ is complete with bounded curvature and $h(1)=g_N$. By the forward and backward uniqueness property of the Ricci flow \cite{ChenZhu06,Kotschwar10} it follows that $h(t)=g_e(t)$.
 
 \subsection{Proof of Theorem \ref{perturbation_thm}} \label{final proofs}

\begin{proof}[Proof of Theorem \ref{perturbation_thm}]
Let $g_{c,Z}=dr^2+r^2 g_1$ be the cone that models the singularity at $z_1$ and $g_{c,exp}=dr^2+r^2 g_1'$ be a cone with  $\riem(g_1')\geq 1$. 

Let $\varepsilon_{link}$ and $\kappa$ be small constants (to be determined in the course of the proof)  such that for $0\leq j \leq 4$,
\begin{equation}
|(\nabla^{g_1})^j (g_1' - g_1)|_{g_1} < \varepsilon_{link},
\end{equation}
on $\mathbb S^{n-1}$, and  $k_Z(r)<\kappa$ for $r\in (0,1]$.

Moreover, let $(N,g_N,f)$ be the expander given by Lemma \ref{lem:expander_pos}, asymptotic to $g_{c,exp}$.

The proof is again similar to the proof of Theorem \ref{stexist}, so we only describe the necessary changes. The approximating sequence $(M_s, G_s)$ is defined as in Subsection \ref{approx}, gluing the expander $(N,g_N,f)$. Then, in Subsection \ref{almost_cone} equation \eqref{error} becomes
\begin{equation}\label{error2}
\begin{aligned}
(&\Phi_s\circ F_s)^* G_s - g_{c,exp}=\\
&=\xi_3(r_s/s^{1/4})(F_s^*g_e(s)- g_{c,exp})+(1-\xi_3(r_s/s^{1/4}))(\phi^* g_Z - g_{c,Z})\\
&\qquad + (1-\xi_3(r_s/s^{1/4}))(g_{c,Z}-g_{c,exp}).
\end{aligned}
\end{equation}
Recall  $\eta_0(g_N)$ given by Theorem \ref{main_thm}. It follows by \eqref{error2} that we may choose $\kappa$, $\varepsilon_{link}$ small and $\Lambda_1$ large (depending on $\eta_0$), such that $(M_s,G_s)\in \mathcal M(\eta_0,\Lambda_1,s)$ for small $s$.

Then, Subsections \ref{limit}-\ref{smoothly_data} carry over unchanged, providing a Ricci flow $(M,g(t))_{t\in (0,T]}$ and a map $\Psi:Z \setminus \{z_1\} \rightarrow M$ such that $\Psi^* g(t)$ converges to $g_Z$ smoothly uniformly away from $z_1$, as  $t\rightarrow 0$.

Now, although Lemma \ref{close_to_expander} is no longer valid, its conclusion does hold for $\varepsilon=0.01$, by the proof of part (1) of Theorem \ref{main_thm}. It follows that Lemma \ref{diam_est} also holds for $(M,g(t))$, hence Subsection \ref{GH} carries over, proving that $g(t)$ converges to $g_Z$ in the Gromov--Hausdorff sense as $t\rightarrow 0$.

\end{proof}

\section{Orbifold quotient expanders and Theorem \ref{orbifold_thm}}\label{orbifold expanders}
  We consider $\mathbb{S}^{n-1}\subset \R^{n}$ and $\Gamma \subset O(n)$ a finite subgroup, acting freely and properly discontinuously on $\mathbb{S}^{n-1}$. Let $\bar{g}$ be a metric on $\mathbb{S}^{n-1}$ with $\text{Rm}(\bar{g})\geq 1$,  but $\text{Rm}(\bar{g}) \not\equiv 1$, which is invariant under the action of $\Gamma$ and thus descends to a metric $g$ on the quotient $\mathbb{S}^{n-1}/\Gamma$. Note that the action of $\Gamma$ thus naturally extends to an isometric action on the cone $(C(\mathbb S^{n-1}), dr^2+r^2\bar{g})$. \\[1ex]
 Let $(N,g_N,f)$ be the unique non-negatively curved gradient Ricci expander $(N,g_N,f)$ given by \cite{Deruelle16}, which is asymptotic to $(C(\mathbb S^{n-1}), dr^2+r^2\bar{g})$, where we assume that $f$ is normalised as in Section \ref{prelim}. By the soliton equation \eqref{soliton_eqn} if follows that $f$ is strictly convex. Let $p_0 \in N$ be the unique point where $f$ attains its minimum, or equivalently $\nabla f(p_0) = 0$. Then, all the level sets $\{f = a\}$ for $a>\min f$ are diffeomorphic to $\mathbb{S}^{n-1}$, and the flow $J_\tau$ of $\nabla f/|\nabla f|^2$ yields natural diffeomorphisms between them. Thus, we may extend the coordinate system at infinity $F$ of Section \ref{prelim} to a diffeomorphism 
  $$F:(0,+\infty)\times \mathbb S^{n-1}\rightarrow N\setminus \{p_0\}, $$ 
 given by $F(r,q)= J_{\frac{r^2 -\Lambda_0^2}{4}} (F(\Lambda_0, q))$.
\\[1ex]
Let us now assume that $\Gamma$ also acts isometrically on $(N,g_N,f)$ and fixes $f$. This implies that the action of $\Gamma$ has to preserve the flow lines of the vector field $\nabla f/|\nabla f|^2$ and thus the action of $\Gamma$ is completely determined by the action on a level set $\{f =a\}$ for $a>\min f$. We will call such an action compatible with the action on $(\mathbb S^{n-1},g)$ if it agrees with the action on the cone $(C(\mathbb S^{n-1}), dr^2+r^2\bar{g})$. In other words,  we call the action of $\Gamma$ compatible if $\gamma \cdot F(r,q) = F(r, \gamma \cdot q)$ for all $\gamma \in \Gamma$. Note that thus the action of $\Gamma$ on the cone uniquely determines the action on $(N,g_N,f)$.

Now, let $\mathcal O$ be an non-compact orbifold with exactly one singular point $p\in \mathcal O$. Then, there is a neighbourhood $U$ of $p$, a neighbourhood  $0\in \tilde U\subset \mathbb R^n$, and a projection $\pi: \tilde U \rightarrow U$ that is invariant under the fixed point free action of a finite subgroup $\Gamma'$ of $O(n)$.

A smooth function $f$ on $\mathcal O$ is a continuous function, smooth on $\mathcal O\setminus \{p\}$, with the property that $\pi^* f$ is smooth. Similarly, a smooth orbifold Riemannian metric $g_{\mathcal O}$ on $\mathcal O$ is a Riemannian metric on $\mathcal O\setminus \{p\}$ with the property that $\pi^* g$ extends smoothly along $0\in \mathbb R^n$.

Since the action of any element of $\Gamma'$ preserves both $\pi^* g$ and $\pi^* f$ it follows that $\nabla^{\pi^* g_{\mathcal O}}\pi^* f$ is a fixed point of the induced action on $T_0\mathbb R^n$. But, since the action is free of fixed points we conclude that $\nabla f|_p = 0$, in the sense that $\nabla^{\pi^* g_{\mathcal O}} \pi^* f|_0=0$.

We call a triple $(\mathcal O,g_{\mathcal O},f)$ an orbifold expander, where $\mathcal O$, $g_{\mathcal O}$ and $f$ are as above, if $\hess_{g_{\mathcal O}} f=\ric(g_{\mathcal O})+\frac{g_{\mathcal O}}{2}$ on $\mathcal O\setminus \{p\}$.

\begin{lemma}\label{quotient}
Let $(\mathcal O,g_{\mathcal O},f)$ be an orbifold expander with positive curvature operator that is asymptotic to the cone $(C(S^{n-1}/\Gamma), dr^2 + r^2 g)$. Suppose that $(S^{n-1}/\Gamma, g)$ is the quotient of $(S^{n-1}, \bar g)$, with $\riem(\bar g)\geq 1$. Then there is a manifold expander $( N,  g_N,  \bar f)$ with positive curvature operator that is asymptotic to the cone $(C(S^{n-1}), dr^2 + r^2  \bar g)$ such that $(\mathcal O,g_{\mathcal O})=(N,g_N) / \Gamma$. It follows that the singularity of the expander is modelled on $\mathbb R^n / \Gamma$.
\end{lemma}

\begin{proof}
It suffices to show that $\mathcal O$ is diffeomorphic to $\mathbb R^n/ \Gamma$. By $\riem(g_{\mathcal O})> 0$ we obtain that $\hess_{g_\mathcal O} f \geq \frac{g_{\mathcal O}}{2}$, hence $\nabla f\not = 0$ on $\mathcal O\setminus \{p\}$. Thus the coordinate system at infinity can be extended to a surjective map
\begin{equation*}
F:(0,+\infty) \times S^{n-1}/\Gamma \rightarrow N\setminus \{p\}. 
\end{equation*}
As in the manifold case, we may assume that this map is related to the flow $J_\tau$ of $\nabla f/|\nabla f|^2$ by
\begin{equation*}
F(r,q)=J_{\frac{r^2-r_0^2}{4}} ( F(r_0,q)),
 \end{equation*} 
 for some $r_0>0$.
 
Observe that $F$ can be deformed to a map $\tilde F:(0,+\infty) \times S^{n-1}/\Gamma \rightarrow \mathcal O\setminus \{p\}$, which extends to a diffeomorphism between $\mathbb R^n/\Gamma$ and $\mathcal O$. To see this, let $\tilde f$ be a smooth function equal to $d_{g_{\mathcal O}}(p,\cdot)^2/4$ near $p$ and to $f$ outside a compact set. Since $\hess_{g_{\mathcal O}} f\geq \frac{g_{\mathcal O}}{2}$ we can arrange so that $\nabla \tilde f \not =0$ in $\mathcal O\setminus \{p\}$.

Now, let $\tilde J_\tau$ be the flow of the field $\nabla \tilde f/|\nabla \tilde f|^2$ and define $\tilde F$ by
\begin{equation*}
\tilde F(r,q)= \tilde J_{\frac{r^2-r_0^2}{4}} (F(r_0,q)).
\end{equation*}
Working on $\pi^* g$ - exponential coordinates  around $\pi^{-1}(p)$ we see that $\tilde F$ is indeed a diffeomorphism.
\end{proof}

\begin{theorem} \label{thm:quotient_expander} Given $(\mathbb{S}^n, \bar{g})$ as above, the action of $\Gamma$ extends to a compatible isometric action on the unique positively curved gradient Ricci expander $(N,g_N,f)$ that is asymptotic to the cone $(C(\mathbb S^{n-1}),dr^2+r^2\bar{g})$. The action fixes $f$ and the only fixed point on $N$ is the critical point $p_0$ of $f$. 
Thus, the quotient space is an expander with exactly one orbifold singularity modelled on $\mathbb R^n/\Gamma$ and is asymptotic to the cone $(C(\mathbb S^{n-1}/\Gamma), dr^2+r^2  g)$.
\end{theorem}

\begin{proof}
We aim to extend Deruelle's proof \cite{Deruelle16} of existence and uniqueness of positively curved gradient expanders to show that the action of $\Gamma$ on the link extends to a compatible, properly discontinuous action on the expander with the claimed properties.  

As in Deruelle, let $(\bar{g}_t)_{0\leq t \leq 1}$ be the (reparametrised) evolution of $\bar{g}$ by volume preserving Ricci flow, such that $\bar{g}_0 = \bar{g}$ and $\bar{g}_1 = \alpha g_{\text{round}}$, where $\alpha = (\text{vol}(\mathbb{S}^{n-1},\bar{g}))^{2/n}$. Since Ricci flow preserves symmetries, $\bar{g}_t$ is invariant under $\Gamma$ for all $t \in [0,1]$. 

Let $(N_t, \tilde g_t, f_t)$ be the unique, positively curved gradient expander asymptotic to the cone $(C(\mathbb S^{n-1}), dr^2+r^2\bar{g}_t)$ obtained by Deruelle. Then, let $p_{0,t} \in N_t$ be the unique point where $\nabla f_t(p_{0,t}) = 0$. 

Note that $(N_1,\tilde g_1,f_1)$ is one of the rotationally symmetric expanders constructed by Bryant (see \cite{CCG}). In this case the action of $\Gamma$ naturally extends to a compatible and properly discontinuous isometric action on $N_1$ which preserves $f_1$ and has only one fixed point $p_{0,1}$.

We want to use an open-closed argument to show that this is true for all $t\in [0,1]$. \\[1ex]
 Recall that $(N_t,\tilde g_t,f_t)$ satisfies the conclusion of the theorem if the following holds: there is an isometric action of $\Gamma$ on $N_t$ with one fixed point,  preserving $f_t$, and the action is compatible with the standard action of $\Gamma$ on the link $(\mathbb{S}^{n-1},\bar{g}_t)$. Note that since the action of $\Gamma$ preserves the level sets of $f_t$, the fixed point has to be $p_{0,t}$. \\[1ex]
{\bf Openness.} Suppose that $(N_t,\tilde g_t,f_t)$ satisfies the conclusion of the theorem. Let $g_{c,t}=dr^2+r^2 \bar g_t$ and $F_t:(0,+\infty)\times \mathbb S^{n-1}\rightarrow N_t$ be the associated coordinate system at infinity, satisfying\

 $$ r^j | (\nabla^{g_{c,t}})^j (F_t^* \tilde g_t-g_{c,t})|_{g_{c,t}} = O(r^{-2}).$$

Then the local uniqueness given in \cite[Theorem 3.7]{Deruelle16} yields an isometric action of $\Gamma$  onto $N_{t'}$, for $t'$ close to $t$. Moreover, there is a diffeomorphism between $N_t$ and $N_{t'}$ identifying this action with the action on $N_t$, so from now on we will work on $N:=N_t$ and  assume that $\tilde g_t, \tilde g_{t'}, f_{t}, f_{t'}$ are defined on $N$. 

This action has a unique fixed point, it preserves $f_t$ by assumption and by the uniqueness statement of Lemma \ref{lem:expander_pos} it follows that it also preserves $f_{t'}$.  We conclude that the fixed point of the action is the critical point $p_{0}$ of  both $f_t$ and $f_{t'}$.

By Theorem \cite[Theorem 3.7]{Deruelle16}, it follows that
\begin{equation}\label{decay1}
r^j | (\nabla^{g_{c,t'}})^j (F_t^* \tilde g_{t'}-g_{c,t'})|_{g_{c,t'}} = O(r^{-2}).
\end{equation}
Observe, however, that the coordinate system $F_t$ is not adapted to the gradient soliton structure of $(N,\tilde g_{t'},f_{t'})$, namely it does not parametrise the level sets of $f_{t'}$. Thus, although the action on $(N_t,\tilde g_t,f_t)$ is compatible to the standard action of $\Gamma$ on $\mathbb S^{n-1}$, it is not immediate that the action on $(N_{t'},\tilde g_{t'},f_{t'})$ is also compatible to the standard action.

For this, we need to construct a diffeomorphism 
$$F_{t'}:[r_0,+\infty)\times \mathbb S^{n-1} \rightarrow \{f_{t'} \geq \frac{r_0^2}{4}\}$$ such that
\begin{enumerate}
\item $f_{t'}(F_{t'}(r,q))=\frac{r^2}{4}$,
\item $r^j | ( \nabla^{g_{c,t'}})^j (F_{t'}^* \tilde g_{t'} - g_{c,t'}  )   |_{g_{c,t'}} = O(r^{-2})$, for all integers $j\geq0$, 
\item $\gamma \cdot F_{t'}(r,q)= F_{t'}(r,\gamma \cdot q)$ where the action on $q$ is the standard action of $\Gamma$ on $\mathbb S^{n-1}$.
\end{enumerate}

Denote by $J_\tau$ the flow of the vector field $\nabla^{\tilde g_t} f_t / |\nabla^{\tilde g_t} f_t|^2$  and by $J'_\tau$ the flow of $\nabla^{\tilde g_{t'}} f_{t'} / |\nabla^{\tilde g_{t'}} f_{t'}|^2$. Since the action leaves both vector fields invariant, it follows that both $J_\tau$ and $J'_\tau$ are equivariant with respect to this action. 

Now fix a large number $r_0>0$. Then,  given any $a\geq \frac{r_0^2}{4}$, define on $\{f_{t} = a\}$ and $\{f_{t'}=a\}$ the Riemannian metrics  
$$(\tilde g_{t'})_{1,\rho} = \rho^{-2}  (J_{\frac{\rho^2}{4}-a})^* \tilde g_{t'} \; \mathrm{and} \;
(\tilde g_{t'})_{2,\rho} = \rho^{-2}  (J'_{\frac{\rho^2}{4}-a})^* \tilde g_{t'}, $$
respectively, for any $\rho\geq r_0$. Here, abusing notation we use $\tilde g_{t'}$ to also denote the restriction of $\tilde g_{t'}$ to the tangent bundle of $\{f_t=\frac{\rho^2}{4}\}$ and $\{f_{t'}=\frac{\rho^2}{4}\}$ respectively.

Note that, from \eqref{decay1}, it follows that 
\begin{equation}\label{conv1}
  (\nabla^{\bar g_{t'}})^j \big( F_t(\rho,\cdot)^* (\tilde g_{t'})_{1,\rho} - \bar g_{t'} \big) = O(\rho^{-2}),
\end{equation}
and from the estimates in \cite[Theorem 3.2]{Deruelle14} 
\begin{equation}\label{conv2}
 (\nabla^{h_a})^j \big( (\tilde g_{t'})_{2,\rho} - h_a\big) = O(\rho^{-2}),
\end{equation}
for some metric $h_a$ on $\{f_{t'}=a\}$, uniformly in $a$.

Moreover, note that
\begin{equation}\label{grad_to_one}
 |\nabla^{\tilde g_{t'}} 2\sqrt{f_{t'}}|_{\tilde g_{t'}}^2=\frac{|\nabla^{\tilde g_{t'}} f_{t'} |^2}{ f_{t'}}=\frac{f_{t'}+R_{min}-R}{f_{t'}}=1+O(f_{t'}^{-1}).
 \end{equation}
 
Now, we claim that  the level set $\{f_{t'}=a\}$ is a graph over $\{f_t=a\}$ via the normal exponential map of $(4a)^{-1}\tilde g_{t'}$,  for each $a\geq \frac{r_0^2}{4}$, if $r_0$ is large. Moreover, the graphing function smoothly converges to zero as $a\rightarrow +\infty$. 
 
To see this, first observe that, as $a\rightarrow +\infty$, any pointed sequence 
$$((4a)^{-1}F_t^* \tilde g_{t'}, x_a), $$ 
with $f_t(x_a)=a$, has a subsequence converging to $(g_{c,t'}, x_\infty)$, with $r(x_\infty)=1$,   in $C^\infty_{loc}$, by \eqref{decay1}. Moreover, under this convergence  $F_t^*(2\sqrt{f_t}/ 2\sqrt a)$ converges to the radial function $r$ of the cone $C(\mathbb S^{n-1})$.

Since
 $$\hess_{(4a)^{-1}\tilde g_{t'}} f_{t'}/a=\ric((4a)^{-1}\tilde g_{t'})+(4a)^{-1}\tilde g_{t'}/2,$$ the curvature decay $\sup_N r^{2+j} |(\nabla^{\tilde g_{t'}})^j \riem|_{\tilde g_{t'}} <+\infty$, implies uniform derivative estimates for $f_{t'}/a$ with respect to $(4a)^{-1}\tilde g_{t'}$ and within bounded distance from $\{f_{t'}=a\}$. Thus, passing to a subsequence $2\sqrt{f_{t'}} /2\sqrt a$ converges smoothly to a limit $r_\infty$ as $a\rightarrow +\infty$, which satisfies $|\nabla^{g_{c,t'}} r_\infty |_{g_{c,t'}}\equiv 1$ due to \eqref{grad_to_one}. Moreover,  $r_\infty\rightarrow 0$ as $r\rightarrow 0$, hence $r_\infty=r=\dist_{g_{c,t'}}(o,\cdot)$, $o$ denoting the tip of the cone. This suffices to prove the claim, since it implies that the level sets of $(4a)^{-1} f_t$ and $(4a)^{-1}f_{t'}$ smoothly converge to each other under this convergence. Note that here we used that the normal exponential map of $(4a)^{-1}\tilde g_{t'}$ over $\{f_t/a=1/4\}$ smoothly converges to the normal exponential map of $g_{c,t'}$ over $\{r=1\}$.
 
 Thus, there is a diffeomorphism, defined via the normal exponential map,  
$$K_a: \{f_{t}=a\} \rightarrow \{f_{t'}=a\},$$ 
satisfying $\gamma\cdot K_a (x) = K_a(\gamma\cdot x)$ for every $x\in \{f_{t}=a\}$ and $\gamma\in \Gamma$.

 Now, as the level sets converge to each other smoothly after scaling, we obtain that 
\begin{equation}
(\nabla^{ (\tilde g_{t'})_{1,2\sqrt a}  })^j (K_a^*(\tilde g_{t'})_{2,\rho} - (\tilde g_{t'})_{1,2\sqrt a} ) \rightarrow 0
\end{equation} 
as $ \frac{\rho^2}{4}\geq a\rightarrow +\infty$, where we also used \eqref{conv2}. Using \eqref{conv1} we obtain
\begin{equation}
(\nabla^{\bar g_{t'}})^j ( F_t(2\sqrt a,\cdot)^*   K_a^* (\tilde g_{t'})_{2,\rho} - \bar g_{t'} ) \rightarrow 0, \label{conv3}
\end{equation} 
as $\frac{\rho^2}{4}\geq a\rightarrow +\infty$.

Define the family of maps given by 
$$F_{t',a}(r,q)= J'_{\frac{r^2}{4}-a}\circ K_a \circ F_t( 2\sqrt a, q) .$$
Observe that the $F_{t',a}$ are equivariant, in the sense that
$$\gamma\cdot F_{t',a}(r,q)=  F_{t'}(r,\gamma\cdot q), $$
for all $\gamma\in \Gamma$ and $q\in\mathbb S^{n-1}$.

Then, we can write 
$$F_{t',a}^* \tilde g_{t'}=F_{t',a}^* (|\nabla^{\tilde g_{t'}} 2\sqrt{f_{t'}}|_{\tilde g_{t'}}^2) dr^2 + r^2  F_t(2\sqrt a,\cdot)^*  K_a^* (\tilde g_{t'})_{2,r}.$$
By \eqref{conv3}, $F_t^* K_a^* (\tilde g_{t'} )_{2,r}$ converges smoothly to $\bar g_{t'}$ as $\frac{r^2}{4}\geq a \rightarrow +\infty$.

Thus,   $F_{t',a}^* \tilde g_{t'}$ is  $C^\infty$-controlled in terms of the metric $dr^2 + r^2 \bar g_{t'}$, uniformly in $a$. Moreover, by \eqref{grad_to_one} 
$$\{f_{t'}\leq \frac{r_1^2}{8}\}\subset F_{t',a}(\{r\leq r_1\}) \subset \{f_{t'}\leq 2 r_1^2\},$$ for $r_1\geq r_0$. 

Taking $a\rightarrow +\infty$, by Arzel\`a--Ascoli a subsequence of $F_{t',a}$ converges to a limit $F_{t'}$. 

Since $F_{t',a}(2\sqrt{b+s},\cdot)=J'_{s}\circ F_{t',a}(2\sqrt b,\cdot)$, it follows that $$F_{t'}(2\sqrt{b+s},\cdot)=J'_{s}\circ F_{t'}(2\sqrt b,\cdot),$$
which implies that requirement (1) above is satisfied. Clearly, (3) is also satisfied since $F_{t',a}$ are equivariant.
 
 Moreover, \eqref{conv3} implies that
$$\lim_{a\rightarrow +\infty} \lim_{r\rightarrow +\infty} F_t^*\circ K_a^* (\tilde g_{t'} )_{2,r}=\bar g_{t'}.$$
This, combined with the estimates of \cite[Theorem 3.2]{Deruelle14} prove (2).

{\bf Closedness.} Let $t_i \rightarrow \bar{t} \in [0,1]$ and assume that $(N_{t_i},\tilde g_{t_i},f_{t_i})$ satisfy the conclusion of the theorem. Consider the sequence of the quotient orbifold expanders $(\mathcal O_i=N_{t_i}/\Gamma, \tilde g_{t_i}, f_{t_i}, p_{0,t_i})$.  Note that for simplicity we use the same notation to denote the metrics and soliton functions in the quotient space. These orbifold expanders have a unique singular point, since the actions of $\Gamma$ on $N_{t_i}$ have a unique fixed point.

The compactness theorem in \cite{Deruelle14} carries over to the orbifold setting, using \cite{Lu01}, to obtain a pointed Cheeger-Gromov limit $(\mathcal O_{\bar t}, \tilde g_{\bar t}, f_{\bar t}, p_{0,\bar t})$, which is an orbifold expander with positive curvature operator. Moreover, $p_{0,\bar t}$ is the unique singular point and the orbifold expander is asymptotic to the cone $(C(\mathbb S^{n-1}/\Gamma), dr^2 + r^2\bar g_{\bar t}/\Gamma)$.

By Lemma \ref{quotient} it follows that there is $(N_{\bar t}, \tilde g_{\bar t}, f_{\bar t}, p_{0,\bar t})$ such that 
$$(\mathcal O_{\bar t}, \tilde g_{\bar t}, f_{\bar t}, p_{0,\bar t})=(N_{\bar t}, \tilde g_{\bar t}, f_{\bar t}, p_{0,\bar t})/\Gamma,$$
and the action on $N_{\bar t}$ is compatible with the standard action of $\Gamma$ on $\mathbb S^{n-1}$.

\end{proof}
 
\begin{remark} We note that the positively curved gradient expander with one isolated orbifold singularity, asymptotic to $(C(\mathbb S^{n-1}/\Gamma), dr^2+r^2g)$ is unique. To see this, note that by Lemma \ref{quotient} the orbifold expander has to be the quotient of a smooth, positively curved expander asymptotic to $(C(\mathbb S^{n-1}), dr^2+r^2\bar{g})$ under the action of $\Gamma$, with a unique fixed point.
\end{remark} 
\begin{remark}
Using the fact that $\Gamma$ has finite characteristic variety, it is possible to employ the continuity argument above to prove the following stronger statement: if $p_0\in N$ is the critical point of the soliton function then there exists an orthonormal basis of $T_{p_0} N$ such that the orthogonal action on $T_{p_0} N$ that is  induced by the isometric action on $N$ is represented by the standard action of $\Gamma$ on $\mathbb R^n$.
\end{remark}

\begin{proof}[Proof of Theorem \ref{orbifold_thm}] The proof is similar to the proof of Theorem \ref{stexist}, so we only describe the necessary changes. For ease of notation, we assume again that there is only one isolated conical singularity at $z_1$. Let $(C(\mathbb{S}^{n-1}/\Gamma), dr^2+r^2g_1)$ be the cone that models the singularity at $z_1$. We denote with $\bar{g}_1$ the lift of $g_1$ to $\mathbb{S}^{n-1}$. Since $(Z,g_Z)$ is asymptotic to $(C(\mathbb{S}^{n-1}/\Gamma), dr^2+r^2g_1)$ , there exists a smooth metric $\bar{g}_Z$ on $(B_1(0)\setminus \{0\}) \subset \R^n$, which is invariant under the natural action of $\Gamma$, such that there is a quotient map $\pi: B_1(0) \rightarrow U$ where $U$ is a neighbourhood of $z_1$ in $Z$ and $\bar{g}_Z = \pi^*g_Z$. Note that this implies that $\bar{g}_Z$ is asymptotic to the cone $(C(\mathbb{S}^{n-1}), dr^2+r^2\bar{g}_1)$ at $0$.

Let $(N,\bar{g}_N,f)$ be the expander given by Lemma \ref{lem:expander_pos}, asymptotic to the cone $(C(\mathbb{S}^{n-1}), dr^2+r^2\bar{g}_1)$. By Theorem \ref{thm:quotient_expander} the action of $\Gamma$ extends to $(N,g_N,f)$. As in Subsection  \ref{approx} we can glue in the orbifold quotient of this expander around $z_0$ into $g_Z$ to obtain an approximating sequence $(M_s,G_s)$ with one orbifold singularity. We can furthermore assume that under $\pi$ this lifts to a corresponding local glueing $(B_1(0), \bar{G}_s)$ of $(N,\bar{g}_N,f)$ into $\bar{g}_Z$. 

By short-time existence for the orbifold Ricci flow, see for example \cite[Section 5.2]{KleinerLott14}, we obtain a solution $(g_s(t))_{t \in [0,T_s]}$ to Ricci flow with an isolated orbifold singularity, starting at $g_s(0)=G_s$. We can arrange this in such a way that the flow lifts under $\pi$ to a smooth Ricci flow $(h_s(t))_{t \in [0,T]}$ on $B_1(0)$, starting at  $\bar{G}_s$. 

Now, all the estimates in Subsections \ref{almost_cone} - \ref{tangent flow} are local, and we can thus apply them to the family $(h_s(t))_{t \in [0,T]}$. Note also that the conclusion of Theorem \ref{main_thm}  holds for $(B_1(0),h_s(t))$. Although $(B_1(0),h_s(t))$ is not complete, all the arguments in the proof of that theorem go through, provided we apply the pseudolocality theorem for orbifolds from \cite{Wang10} to $(M_s,g_s(t))$, to obtain the necessary curvature estimates in the conical region.

Projecting under $\pi$ we obtain the corresponding estimates for $(g_s(t))_{t \in [0,T]}$. In particular, as in Corollary \ref{corollary}, we obtain a uniform existence time $T$ for $g_s(t)$ and the curvature bound  
$$\max_{M_s} |\riem(g_s(t))|_{g_s(t)}\leq C/t.$$
Thus, by the compactness theorem for orbifold Ricci flow in \cite{Lu01}, there exists a limit Ricci flow $(g(t))_{t \in [0,T]}$ with an isolated orbifold singularity and the claimed properties.
\end{proof}


\providecommand{\bysame}{\leavevmode\hbox to3em{\hrulefill}\thinspace}

\end{document}